\documentclass[article,12pt]{amsart}

\usepackage{amsfonts}
\usepackage{amsmath}
\usepackage{graphicx}
\usepackage{color}
\usepackage{epsfig}

\numberwithin{equation}{section}
\theoremstyle{plain}
\newtheorem{thm}{Theorem}[section]
\newtheorem{rem}{Remark}[section]
\newtheorem{prop}{Proposition}[section]

\newtheorem{lem}{Lemma}[section]
\newtheorem{defi}[lem]{Definition}

\newcommand{\dC}{\mathbb{C}}
\newcommand{\dI}{\mathbb{I}}
\newcommand{\dN}{\mathbb{N}}
\newcommand{\dZ}{\mathbb{Z}}

\newcommand{\dE}{\mathbb{E}}
\newcommand{\dP}{\mathbb{P}}

\newcommand{\cH}{\mathcal{H}}

\newcommand{\cD}{\mathcal{D}}

\newcommand{\cR}{\mathcal{R}}
\newcommand{\cS}{\mathcal{S}}
\newcommand{\cHz}{\cH_0}
\newcommand{\cHa}{\cH_1}
\newcommand{\veps}{\varepsilon}
\newcommand{\wh}{\widehat}
\newcommand{\whveps}{\wh{\varepsilon}}
\newcommand{\whvepsq}{\wh{\varepsilon}^{~ 2}}
\newcommand{\hsp}{\hspace{0.5cm}}
\newcommand{\whk}{\wh{K}_{T}}

\newcommand{\liml}{\overset{\cD}{\longrightarrow}}
\newcommand{\egl}{\overset{\cD}{=}}
\newcommand{\limp}{\overset{\dP}{\longrightarrow}}
\newcommand{\limT}{\lim_{T\, \rightarrow\, \infty}}
\newcommand{\dd}{\, \mathrm{d}}
\newcommand{\cvgps}{\hspace{0.3cm} \textnormal{a.s.}}

\newcommand{\sigeps}{\sigma^2_{\veps}}
\newcommand{\sige}{\sigma^2_{\eta}}

\email{frederic.proia@univ-angers.fr}

\begin{document}

\title[Stationarity vs integration in the AR process]
{Stationarity against integration in the autoregressive process with polynomial trend
\vspace{2ex}}
\address{Laboratoire Angevin de REcherche en MAth\'ematiques (UMR 6093). Universit\'e d'Angers, D\'epartement de math\'ematiques, Facult\'e des Sciences, 2 Boulevard Lavoisier, 49045 Angers cedex, France.}
\author{Fr\'ed\'eric Pro\"ia}
\thanks{}

\begin{abstract}
We tackle the stationarity issue of an autoregressive path with a polynomial trend, and we generalize some aspects of the LMC test, the testing procedure of Leybourne and McCabe. First, we show that it is possible to get the asymptotic distribution of the test statistic under the null hypothesis of trend-stationarity as well as under the alternative of nonstationarity, for any polynomial trend of order $r$. Then, we explain the reason why the LMC test, and by extension the KPSS test, does not reject the null hypothesis of trend-stationarity, mistakenly, when the random walk is generated by a unit root located at $-1$. We also observe it on simulated data and we correct the procedure. Finally, we describe some useful stochastic processes that appear in our limiting distributions.
\end{abstract}

\maketitle

\noindent \textbf{Notations.} In all the paper, we define $L$ as the lag operator with the convention that $L^0 = I$. In addition, $k_{T} = k/T$ is the renormalization of any $k \in \dN$, and $\dI$ designates the indicator function. We will always consider that $0 < \tau \leq 1$ and that $[T \tau]$ denotes the integer part of $T \tau$. To lighten the notations, we will usually refer to the corresponding vector by removing the implicit subscript on the variable. For example, $\veps^{\, \prime} = ( \veps_1 \hsp \hdots \hsp \veps_{T})$ where $\veps^{\, \prime}$ is the transpose of $\veps$.

\section{A consistent test for a unit root}
\label{SecSingle}

%%%%%%%%%%%%%%%%%%%%%%%%%%%%%%%%%%%%%%%%%%%%%%%%%%%%%%%%%%%%%%%%%%%%%%%%%%%%%%%%

We consider the autoregressive process of order $p$ on $\dZ$ with a polynomial trend of order $r$, driven by a random walk and an additive error. For an observed path of size $T$, we investigate the model given, for all $1 \leq t \leq T$, by
\begin{equation}
\label{ModSingle}
\Theta(L) X_{t} = ( \alpha_0 + \alpha_1 t_{T} + \hdots + \alpha_{r} t_{T}^{r})\, \dI_{\{ \kappa\, \neq\, 0 \}} + S_{t}^{\eta} + \veps_{t}
\end{equation}
where, for all $z \in \dC$, $\Theta(z) = 1 - \theta_1 z - \hdots - \theta_{p}\, z^{p}$ is an autoregressive polynomial having all its zeroes outside the unit circle, where, for any $\vert \rho \vert = 1$,
\begin{equation}
\label{ModRandWalk}
S^{\eta}_{t} = \rho S^{\eta}_{t-1} + \eta_{t}
\end{equation}
is a random walk starting from $S_0^{\eta} = 0$, and where $(\veps_{t})$ and $(\eta_{t})$ are uncorrelated white noises of variance $\sigma_{\veps}^2 > 0$ and $\sigma_{\eta}^2 \geq 0$, respectively. From now on, white noises are to be interpreted in the strong sense, that is as sequences of independent and identically distributed random variables. For the sake of simplicity, we consider that $X_{-p+1} = \hdots = X_{-1} = 0$. We also normalize the known part of the trend, by selecting $t_{T} = t/T$, to simplify the treatment of the projections, as we will see in the technical proofs. The order of the polynomial trend is $r$, but we will also take account of the case where no trend is introduced in \eqref{ModSingle}. We switch from one situation to another by selecting $\kappa \neq 0$ or $\kappa = 0$. Our objective is to establish a testing procedure for
\begin{equation*}
\cH_0 : ``\sige = 0" \hsp \hsp \text{against} \hsp \hsp \cH_1 : ``\sige > 0".
\end{equation*}
One can observe that \eqref{ModSingle} is a trend-stationary process under the null $\cHz$, since the process $(S^{\eta}_{t})$ is almost surely zero, and an integrated process of order 1 under the alternative $\cHa$. Hence, evaluating $\cHz$ against $\cHa$ is equivalent to testing stationarity against integration in the stochastic part of the process. In this context, our work is a generalization of the procedure of \cite{LeybourneMcCabe94}, shortened \textit{LMC test} in all the sequel. In their original paper, they propose to make use of the maximum likelihood estimator of $\theta$ on a given path of size $T$ and to estimate the trend parameters using a least squares methodology on the residual process. Then, they build a test statistic and establish its behavior under the null hypothesis of stationarity for specific trends (none, constant or linear). Under $\cH_1$, they show that the test statistic diverges with rate $T$, and that it is possible to get its correctly renormalized asymptotic distribution. In the simple case where $p=0$, \cite{NabeyaTanaka88} had already investigated the founding principles of this strategy. This restriction seems nevertheless far from the reality of time series since all correlation phenomenon has disappeared. Earlier, \cite{NyblomMakelainen83}, \cite{Nyblom86} and \cite{LeybourneMcCabe89} had already taken an interest in such test statistics, for closely related models. The procedure of \cite{KwiatkowskiPhillipsSchmidtShin92}, shortened from now on \textit{KPSS test}, translates any correlation in the residual process, to avoid any preliminary estimation of $p$ and $\theta$. Their test statistic (described later in Remark \ref{RemKPSS}) is shown to reach the same asymptotic distribution but, as a long-run variance has to be estimated instead, there is a truncation at a lag $\ell$ such that $\ell = \ell(T) \rightarrow \infty$ to ensure consistency, and the divergence under $\cH_1$ occurs with rate $T/\ell = o(T)$. One can accordingly expect that the LMC procedure will be more powerful to discriminate $\cH_1$, and such observations are made in \cite{LeybourneMcCabe94}. However, the true value of $p$ is needed and all flexibility is sacrificed, contrasting with the KPSS procedure. The stationarity of time series being a contemporary issue, it is not surprising to find an abundant literature on empirical studies, anomalies detection or improvements brought to these strategies: let us mention \cite{SaikkonenLuukkonen93}, \cite{LeybourneMcCabe99}, \cite{NewboldLeybourneWohar01}, or \cite{Muller05}, \cite{HarrisLeybourneMcCabe07}, \cite{DeJongAmsterSchmidt07}, \cite{PelagattiSen09} and all associated references, without completeness. First we will show that in the context of the LMC test, it is possible to get the asymptotic distribution of the test statistic under $\cH_0$ as well as under $\cH_1$, for any polynomial trend of order $r$. Then, we will explain, and we will observe it on some straightforward simulated data, the reason why the LMC test -- and by extension the KPSS test -- does not reject the null hypothesis of trend-stationarity, mistakenly, when the random walk is generated by a unit root located at $-1$. We have widely been inspired by the calculation methods of \cite{Phillips87}, \cite{KwiatkowskiPhillipsSchmidtShin92} and \cite{LeybourneMcCabe94}, themselves relying on the Donsker's invariance principle and the Mann-Wald's theorem, that we will also recall. Finally, we will describe some useful stochastic processes that appear in our limiting distributions, and we will prove our results.

\medskip

The case $\vert \rho \vert < 1$ corresponds to a trend-stationary process both under $\cHz$ and under $\cHa$, it is consequently not of interest as part of this paper. Combining \eqref{ModSingle} and \eqref{ModRandWalk}, the model under $\cHa$ is
\begin{equation}
\label{ModSingleH1}
\Theta(L) X_{t} = (\alpha_0 + \alpha_1 t_{T} + \hdots + \alpha_{r}\, t_{T}^{r})\, \dI_{\{ \kappa\, \neq\, 0 \}} + \sum_{k=1}^{t} \rho^{t-k} \eta_{k} + \veps_{t}
\end{equation}
where the source of the stochastic nonstationarity of $(X_{t})$ is
\begin{equation}
\label{RWSingle}
S_{t}^{\eta} = \sum_{k=1}^{t} \rho^{t-k} \eta_{k}
\end{equation}
which is the partial sum process of $(\eta_{t})$ when $\rho=1$. First,
\begin{eqnarray*}
\Theta(L) (I - \rho L) X_{t} & = & (I - \rho L)(\alpha_0 + \alpha_1 t_{T} + \hdots + \alpha_{r}\, t_{T}^{r})\, \dI_{\{ \kappa\, \neq\, 0 \}} + (I - \rho L) (S_{t}^{\eta} + \veps_{t}) \\
 & = & (\alpha^{*}_0 + \alpha^{*}_1 t_{T} + \hdots + \alpha^{*}_{r}\, t_{T}^{r})\, \dI_{\{ \kappa\, \neq\, 0 \}} + \eta_{t} + (I - \rho L) \veps_{t}
\end{eqnarray*}
where $\alpha^{*}_0, \alpha^{*}_1, \hdots, \alpha^{*}_{r}$ are easily identifiable (\textit{e.g.} $\alpha^{*}_{r} = 0$ when $\rho = 1$) and the process $(\eta_{t} + (I - \rho L) \veps_{t})$ is second-order equivalent in moments to an MA$(1)$ residual, as it is explained in \cite{KwiatkowskiPhillipsSchmidtShin92}. We obtain the integrated model given, for all $1 \leq t \leq T$, by
\begin{equation}
\label{ModSingleDiff}
\Theta(L) (I - \rho L) X_{t} = (\alpha^{*}_0 + \alpha^{*}_1 t_{T} + \hdots + \alpha^{*}_{r}\, t_{T}^{r})\, \dI_{\{ \kappa\, \neq\, 0 \}} + (I + \beta L)\, \xi_{t}
\end{equation}
where $(\xi_{t})$ is a white noise of variance $\sigma_{\xi}^2$ depending on the so-called \textit{signal-to-noise ratio} $\sigma_{\eta}^2/\sigma_{\veps}^2$. For the generating process \eqref{ModSingleDiff}, we build a consistent estimator of $\theta$ (see Remark \ref{RemEstT} below), and we consider the residual process
\begin{equation}
\label{Ystar}
\check{X}_{t} = X_{t} - \check{\theta}_1 X_{t-1} - \hdots - \check{\theta}_{p} X_{t-p}.
\end{equation}
Note that under $\cHa$, $\vert \beta \vert < 1$, implying that the differentiated process is causal and invertible. On the other hand, $\vert \beta \vert = 1$ under $\cHz$ and the process is not invertible.

\begin{rem}
\label{RemEstT}
The consistency of $\check{\theta}_{T}$ is a crucial issue of the study. Let $\Lambda_{r,\rho}(L)$ be the operator defined as
$$
\Lambda_{r,\rho}(L) = (I - \rho L) (I - L)^{r}.
$$
It follows that
\begin{eqnarray*}
\Lambda_{r,\rho}(L)\, \Theta(L) X_{t} & = & (I - L)^{r}\, (\alpha^{*}_0 + \alpha^{*}_1 t_{T} + \hdots + \alpha^{*}_{r} t_{T}^{r})\, \dI_{\{ \kappa\, \neq\, 0 \}} + (I - L)^{r}\, (I + \beta L)\, \xi_{t} \\
 & = & \mu^{*} + \Phi(L)\, \xi_{t}
\end{eqnarray*}
where $\mu^{*}$ is easily identifiable ($\mu^{*}=0$ when $\rho=1$ or $\kappa=0$) and $\Phi$ is a moving average polynomial of order $r+1$. Now, let
$$
Y_{t} = \Lambda_{r,\rho}(L)\, X_{t}.
$$
Clearly, $\Theta(L) Y_{t} = \mu^{*} + \Phi(L)\, \xi_{t}$ implying that $(Y_{t})$ is a causal ARMA$(p,r+1)$ process having a potentially nonzero intercept. Under $\cH_1$, $\Phi$ has $r$ unit roots and its last zero is outside the unit circle (since $\vert \beta \vert < 1$). Under $\cH_0$, $\Phi$ has $r+1$ unit roots. In both situations, Theorem 2.1 of \cite{Potscher91} ensures that a pseudo-MLE is consistent for $(\theta, \beta)$ treating $(\xi_{t})$ as  a Gaussian noise, whereas $\mu^{*}$ is easily estimated as intercept of a stationary ARMA process. Nevertheless, only the AR part of the process is of interest for us, and a faster method is worth considering. The causality of $\Theta$ implies that there exists a causal representation
$$
Y_{t} = \Theta^{-1}(L)\left( \mu^{*} + \Phi(L)\, \xi_{t} \right) = \nu^{*} + \sum_{k=0}^{\infty} \psi_{k}\, \xi_{t-k}
$$
such that, according to Chapter 7 of \cite{BrockwellDavis96}, the sample autocovariance function $\wh{\gamma}_{T}$ of $(Y_{t} - \nu^{*})$ is a consistent estimator of its autocovariance function $\gamma_{Y}$. Using a Yule-Walker approach, for all $h \in \{ r+2, \hdots, p+r+1 \}$,
$$
\gamma_{Y}(h) = \sum_{k=1}^{p} \theta_{k} \gamma_{Y}(h-k).
$$
Hence, a consistent estimator of $\theta$ may be obtained \textit{via} $\wh{\gamma}_{T}$. The selection of $\rho$ will be widely discussed in Section \ref{SecEmpi}.
\end{rem}

As a result of the previous remark, it makes sense to estimate $\alpha$ under $\cHz$ using a least squares methodology in the model given by
\begin{equation}
\label{ModSingleH0}
\check{X}_{t} = (\alpha_{0} + \alpha_1 t_{T} + \hdots + \alpha_{r}\, t_{T}^{r})\, \dI_{\{ \kappa\, \neq\, 0 \}} + \check{\veps}_{t}
\end{equation}
where $(\check{\veps}_{t})$ is the residual process coming from the estimation of $\theta$. A second-order residual set $(\whveps_{t})$ is then built \textit{via}
\begin{equation}
\label{EstRes}
\whveps_{t} = \check{X}_{t} - (\wh{\alpha}_0 + \wh{\alpha}_1 t_{T} + \hdots + \wh{\alpha}_{r}\, t_{T}^{r})\, \dI_{\{ \kappa\, \neq\, 0 \}}
\end{equation}
where $\wh{\alpha}_{T}$ is the least squares estimator of $\alpha$ in the model \eqref{ModSingleH0}. Let the partial sum processes of $(\whveps_{t})$ and $(\whvepsq_{t})$ be defined as
\begin{equation}
\label{SumEstRes}
S_{t} = \sum_{k=1}^{t} \whveps_{k} \hsp \hsp \text{and} \hsp \hsp Q_{t} = \sum_{k=1}^{t} \whvepsq_{k}.
\end{equation}
Finally, consider the test statistic
\begin{equation}
\label{TestStat}
\whk = \frac{1}{T Q_{T}}\, \sum_{t=1}^{T} S_{t}^{\, 2}.
\end{equation}

\begin{rem}
\label{RemKPSS}
The test statistic of the KPSS procedure is very close to $\whk$. The main difference is that $(\veps_{t})$ satisfies some weaker assumptions including correlation (see \cite{KwiatkowskiPhillipsSchmidtShin92}), leading to $p=0$ and no parameter $\theta$ to estimate. In return, a long-run variance defined as
$$
\sigma^2 = \limT \frac{1}{T}\, \dE[S_{T}^2]
$$
has to be estimated using a truncation method. The test statistic is
$$
\whk = \frac{1}{T^2\, \wh{\sigma}_{T}^2}\, \sum_{t=1}^{T} S_{t}^{\, 2}
$$
and corresponds to \eqref{TestStat} when $\wh{\sigma}_{T}^2 = Q_{T}/T$, that is when the long-run variance is estimated as a white noise variance.
\end{rem}

We now establish the asymptotic behavior of $\whk$ under $\cHz$. The stochastic processes appearing in our limiting distributions are described in the next section.
\begin{thm}
\label{ThmStatH0}
Assume that $\sigma_{\eta}^2 = 0$. Then, for $\kappa \neq 0$, we have the weak convergence
\begin{equation*}
\whk \liml \int_0^1 B_{r}^{\, 2} (s) \dd s
\end{equation*}
where $(B_{r}(t),\, t \in [0,1])$ is the generalized Brownian bridge of order $r$. In addition, for $\kappa = 0$, we have the weak convergence
\begin{equation*}
\whk \liml \int_{0}^{1} W^{\, 2}(s) \dd s
\end{equation*}
where $(W(t),\, t \in [0,1])$ is the standard Wiener process.
\end{thm}

\medskip

In the following theorem, we show that $\whk$ diverges under $\cHa$ for $\rho=1$ with rate $T$ and we study the asymptotic behavior of the test statistic correctly renormalized. We also show that it decreases to zero under $\cHa$ for $\rho=-1$.
\begin{thm}
\label{ThmStatH1}
Assume that $\sigma_{\eta}^2 > 0$. Then, for $\kappa \neq 0$ and $\rho=1$, we have the weak convergence
\begin{equation*}
\frac{\whk}{T} \liml \frac{\int_0^1 C_{r,\, 1}^{\, 2}(s) \dd s}{\int_0^1 W_{r,\, 0}^{\, 2}(s) \dd s}
\end{equation*}
where $(C_{r,\, 1}(t),\, t \in [0,1])$ is the integrated Brownian bridge of order $r \times 1$ and $(W_{r,\, 0}(t),\, t \in [0,1])$ is the detrended Wiener process of order $r \times 0$. In addition, for $\kappa = 0$, we have the weak convergence
\begin{equation*}
\frac{\whk}{T} \liml \frac{\int_0^1 W^{(1)\, 2}(s) \dd s}{\int_0^1 W^{\, 2}(s) \dd s}
\end{equation*}
where $(W^{(1)}(t),\, t \in [0,1])$ is the integrated Wiener process of order 1 and $(W(t),\, t \in [0,1])$ is the standard Wiener process. Finally, for $\rho=-1$,
\begin{equation*}
\whk \limp 0.
\end{equation*}
\end{thm}

\medskip

The situation where $\rho=-1$ is the cause of a number of complications as we will see in the associated proofs, that is the reason why we limit ourselves to stipulate the convergence of $\whk$ to zero in the general case. However, in the particular case where $\kappa = 0$, we reach the following result.

\begin{prop}
\label{PropStatH1b}
Assume that $\sigma_{\eta}^2 > 0$. Then, for $\kappa = 0$ and $\rho=-1$, we have the weak convergence
\begin{equation*}
T\, \whk \liml \frac{2 \sigeps \int_0^1 W_{\veps}^{\, 2}(s) \dd s + \sige \int_0^1 W_{\eta}^{\, 2}(s)}{2 \sige \int_0^1 W_{\eta}^{\, 2}(s)}
\end{equation*}
where $(W_{\veps}(t),\, t \in [0,1])$ and $(W_{\eta}(t),\, t \in [0,1])$ are independent standard Wiener processes.
\end{prop}

\medskip

One can notice that this is the only situation in which $(\veps_{t})$ and $(\eta_{t})$ simultaneously play a role in the asymptotic behavior, this explains why we had to make such a decomposition into $W_{\veps}(t)$ and $W_{\eta}(t)$. As a matter of fact, under $\cH_0$, $(\veps_{t})$ is the only perturbating process whereas under $\cH_1$ with $\rho=1$, $(\veps_{t})$ is dominated by $(\eta_{t})$. We are pretty convinced, on the basis of a simulation study, that it is possible to find an identifiable limiting distribution for $T \whk$ when $\kappa \neq 0$ and $\rho=-1$. However, we have not reached the explicit expression in this work because of complications due to the phenomenon of compensation in the invariance principles, and calculations very hard to conduct. This could form an objective for a future study.

\begin{proof}
Theorems \ref{ThmStatH0}--\ref{ThmStatH1} and Proposition \ref{PropStatH1b} are proved in Section \ref{SecProof}.
\end{proof}

\begin{rem}
It is also possible to extend the whole results to the multi-integrated processes under the alternative, such as ARI processes having more than one unit root. In model \eqref{ModSingle}, the random walk $(S_{t}^{\eta})$ is now itself generated by a random walk, and so on up to $d \geq 0$ positive unit roots. Then, weak convergences in Theorem \ref{ThmStatH1} become
\begin{equation*}
\frac{\whk}{T} \liml \frac{\int_0^1 C_{r,\, d}^{\, 2}(s) \dd s}{\int_0^1 W_{r,\, d-1}^{\, 2}(s) \dd s} \hsp \text{and} \hsp \frac{\whk}{T} \liml \frac{\int_0^1 W^{(d)\, 2}(s) \dd s}{\int_0^1 W^{(d-1)\, 2}(s) \dd s},
\end{equation*}
respectively for $\kappa \neq 0$ and $\kappa = 0$. For $d \geq 1$ negative unit roots, we still reach the convergence
\begin{equation*}
\whk \limp 0.
\end{equation*}
Such results may be useful to produce a statistical testing procedure concerning the integration order $d$ of the generating process of an observed path and/or to check the true value of $r$.
\end{rem}

On Figure \ref{FigDistr}, we have represented the asymptotic distribution of $\whk$ under $\cH_0$ for $\kappa=0$, then for $\kappa \neq 0$ and $r \in \{0, \hdots, 4\}$, using Monte-Carlo experiments.

\begin{figure}[h!]
\centering
\includegraphics[width=14cm]{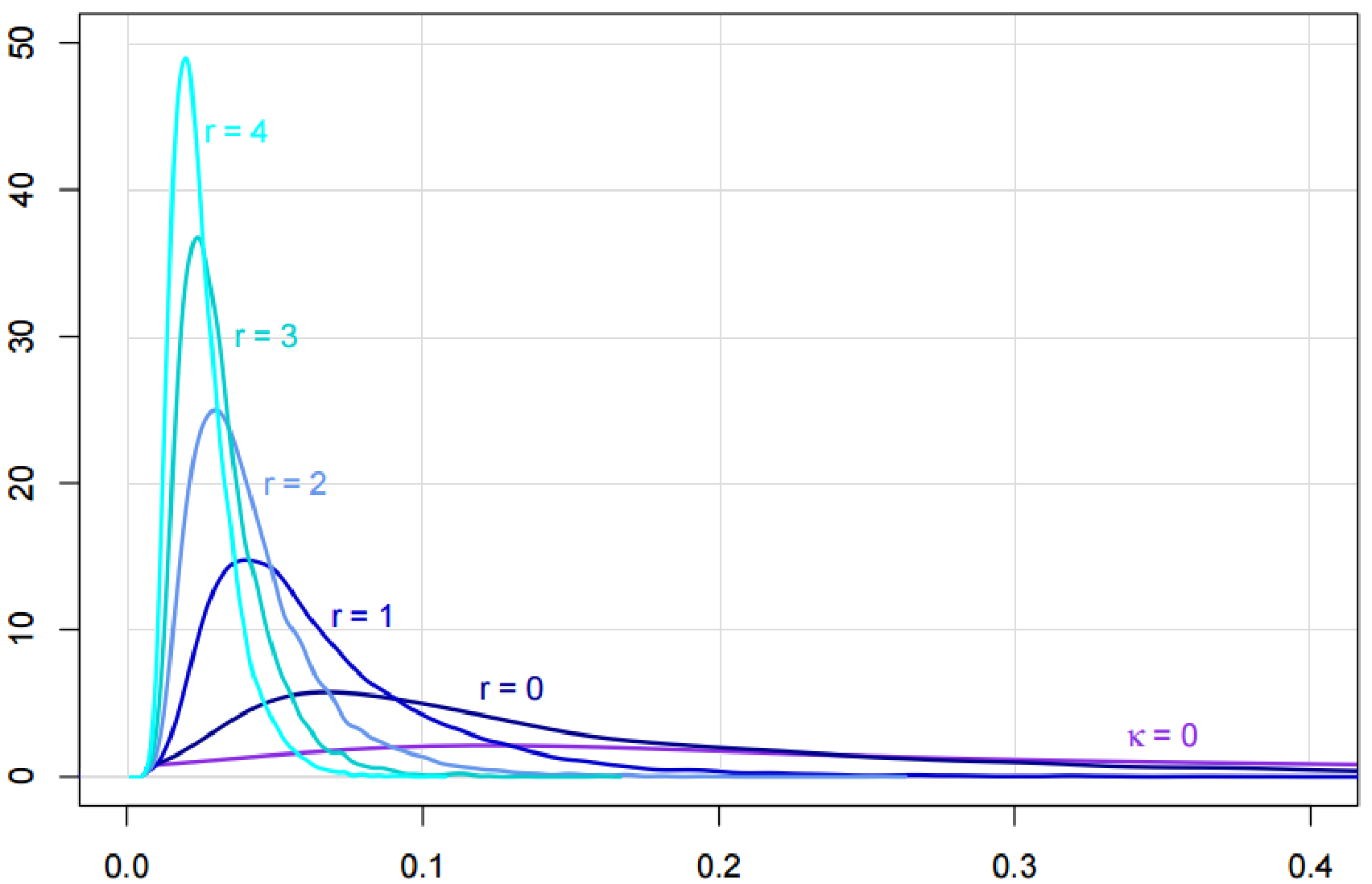}
\caption{\small Asymptotic distribution of $\whk$ under $\cH_0$ for $\kappa=0$, then for $\kappa \neq 0$ and $r \in \{0, \hdots, 4\}$, using Monte-Carlo experiments.}
\label{FigDistr}
\end{figure}

%%%%%%%%%%%%%%%%%%%%%%%%%%%%%%%%%%%%%%%%%%%%%%%%%%%%%%%%%%%%%%%%%%%%%%%%%%%%%%%%

\section{Some useful stochastic processes}
\label{SecProc}

%%%%%%%%%%%%%%%%%%%%%%%%%%%%%%%%%%%%%%%%%%%%%%%%%%%%%%%%%%%%%%%%%%%%%%%%%%%%%%%%

Throughout the study, we deal with some stochastic processes, built from the standard Wiener process $(W(t),\, t \in [0,1])$ that we are now going to introduce. In all definitions, we consider that $d,r \in \dN$.
\begin{defi}[Integrated Wiener Process]
\label{DefIWP}
The process given, for $t \in [0,1]$, by
\begin{equation*}
W^{(d)}(t) = \int_0^{t} \int_0^{s_{1}} \hdots \int_0^{s_{d-1}} W(s_{d}) \dd s_{d} \hdots \dd s_{1}
\end{equation*}
is called a ``integrated Wiener process of order $d$" in the whole paper. By convention, $W^{(0)}(t) \equiv W(t)$.
\end{defi}

\noindent For example,
\begin{equation*}
W^{(1)}(t) = \int_{0}^{t} W(s) \dd s \hsp \text{and} \hsp W^{(2)}(t) = \int_{0}^{t} \int_{0}^{s} W(u) \dd u \dd s.
\end{equation*}

\begin{defi}[Generalized Brownian Bridge]
\label{DefGBB}
The process given, for $t \in [0,1]$, by
\begin{equation*}
B_{r}(t) = h_{r}(W)(t)
\end{equation*}
where $h_{r}$ is an application from $C([0,1])$ into $C([0,1])$ given by formula (8) in \cite{MacNeill78}, is called a ``generalized Brownian bridge of order $r$" in the whole paper.
\end{defi}

\begin{defi}[Integrated Brownian Bridge]
\label{DefIBB}
The process given, for $t \in [0,1]$, by
\begin{equation*}
C_{r,\, d}(t) = h_{r}(W^{(d)})(t)
\end{equation*}
is called a ``integrated Brownian bridge of order $r \times d$" in the whole paper. By convention, $C_{r,\, 0}(t) \equiv B_{r}(t)$.
\end{defi}

\begin{defi}[Detrended Wiener Process]
\label{DefDWP}
The process given, for $t \in [0,1]$, by
\begin{equation*}
W_{r,\, d}(t) = \frac{\dd C_{r,\, d+1}(t)}{\dd t}
\end{equation*}
is called a ``detrended Wiener process of order $r \times d$" in the whole paper. It is explicitly defined as
\begin{equation*}
W_{r,\, d}(t) = W^{(d)}(t) - P_{d}^{\, \prime}(1) M^{-1} \Lambda(t)
\end{equation*}
where the nonsingular matrix $M$ satisfies $M_{i j} = 1/(i+j-1)$ for all $1 \leq i,j \leq r+1$, $\Lambda(t) = \begin{pmatrix} 1 & t & \hdots & t^{r} \end{pmatrix}^{\prime}$, and where
\begin{equation}
\label{Pd}
P_{d}^{\, \prime}(t) = \begin{pmatrix} W^{(d)}(t) & \displaystyle \int_0^{t} s\, W^{(d-1)}(s) \dd s & \hdots & \displaystyle \int_0^{t} s^{r}\, W^{(d-1)}(s) \dd s \end{pmatrix}.
\end{equation}
\end{defi}

\noindent Let us illustrate these definitions on the standard cases $r = \{ 0,\, 1 \}$ and $d=0$. According to Definition \ref{DefGBB} and formula (8) in \cite{MacNeill78}, for $t \in [0,1]$,
\begin{equation*}
B_0(t) = h_{0}(W)(t) = W(t) - t W(1)
\end{equation*}
which is the usual ``Brownian bridge". It follows from Definitions \ref{DefIBB} and \ref{DefDWP} that
\begin{equation*}
C_{0,\, 1}(t) = h_{0}(W^{(1)})(t) = \int_0^{t} W(s) \dd s - t \int_0^1 W(s) \dd s
\end{equation*}
and that
\begin{equation*}
W_{0,\, 0}(t) = \frac{\dd C_{0,\, 1}(t)}{\dd t} = W(t) - \int_0^1 W(s) \dd s
\end{equation*}
which is the usual ``demeaned Wiener process". Similarly, for $r=1$,
\begin{equation*}
B_1(t) = h_{1}(W)(t) = W(t) + t (2 - 3t) W(1) - 6 t (1-t) \int_0^1 W(s) \dd s
\end{equation*}
is the ``second-level Brownian bridge", leading to
\begin{equation*}
C_{1,\, 1}(t) = \int_0^{t} W(s) \dd s + t (3t - 4) \int_0^1 W(s) \dd s + 6 t (1-t) \int_0^1 s\, W(s) \dd s.
\end{equation*}
Finally,
\begin{equation*}
W_{1,\, 0}(t) = \frac{\dd C_{1,\, 1}(t)}{\dd t} = W(t) + (6t - 4) \int_0^1 W(s) \dd s + ( 6 - 12 t) \int_0^1 s\, W(s) \dd s
\end{equation*}
is the standard ``detrended Wiener process".

%%%%%%%%%%%%%%%%%%%%%%%%%%%%%%%%%%%%%%%%%%%%%%%%%%%%%%%%%%%%%%%%%%%%%%%%%%%%%%%%

\section{A corrected test adapted to the negative unit root}
\label{SecEmpi}

%%%%%%%%%%%%%%%%%%%%%%%%%%%%%%%%%%%%%%%%%%%%%%%%%%%%%%%%%%%%%%%%%%%%%%%%%%%%%%%%

The empirical power of the KPSS and LMC procedures has been widely studied in the literature (see Section \ref{SecSingle} for references). For $\rho=1$, the 	improvements that we described in this paper (for any $r$ and $d$) are mainly theoretical. On the other hand, we thought useful to conduct an empirical study for $\rho=-1$, because in this case it is not only a matter of generalization but also a matter of \textit{correction} of the existing procedures. To motivate the study, we have represented on Figures \ref{FigExSimK0}--\ref{FigExSimR2} below some examples of simulations according to \eqref{ModSingleH1} under $\cH_0 : ``\sige = 0"$, under $\cH_1^{+} : ``\sige > 0 \text{ and } \rho=1"$ and under $\cH_1^{-} : ``\sige > 0 \text{ and } \rho=-1"$, using the configurations indicated in the captions. Clearly, a visual investigation is required to decide whether $\rho=1$ or $\rho=-1$ is the most likely alternative, which is a crucial point to stationarize the process. In the whole experiments, the quantiles of the limit distribution of the test statistic under the null, depending on $\kappa$ and $r$, have been taken from Table 2 of \cite{MacNeill78}.

\begin{figure}[h!]
\centering
\includegraphics[width=4.9cm]{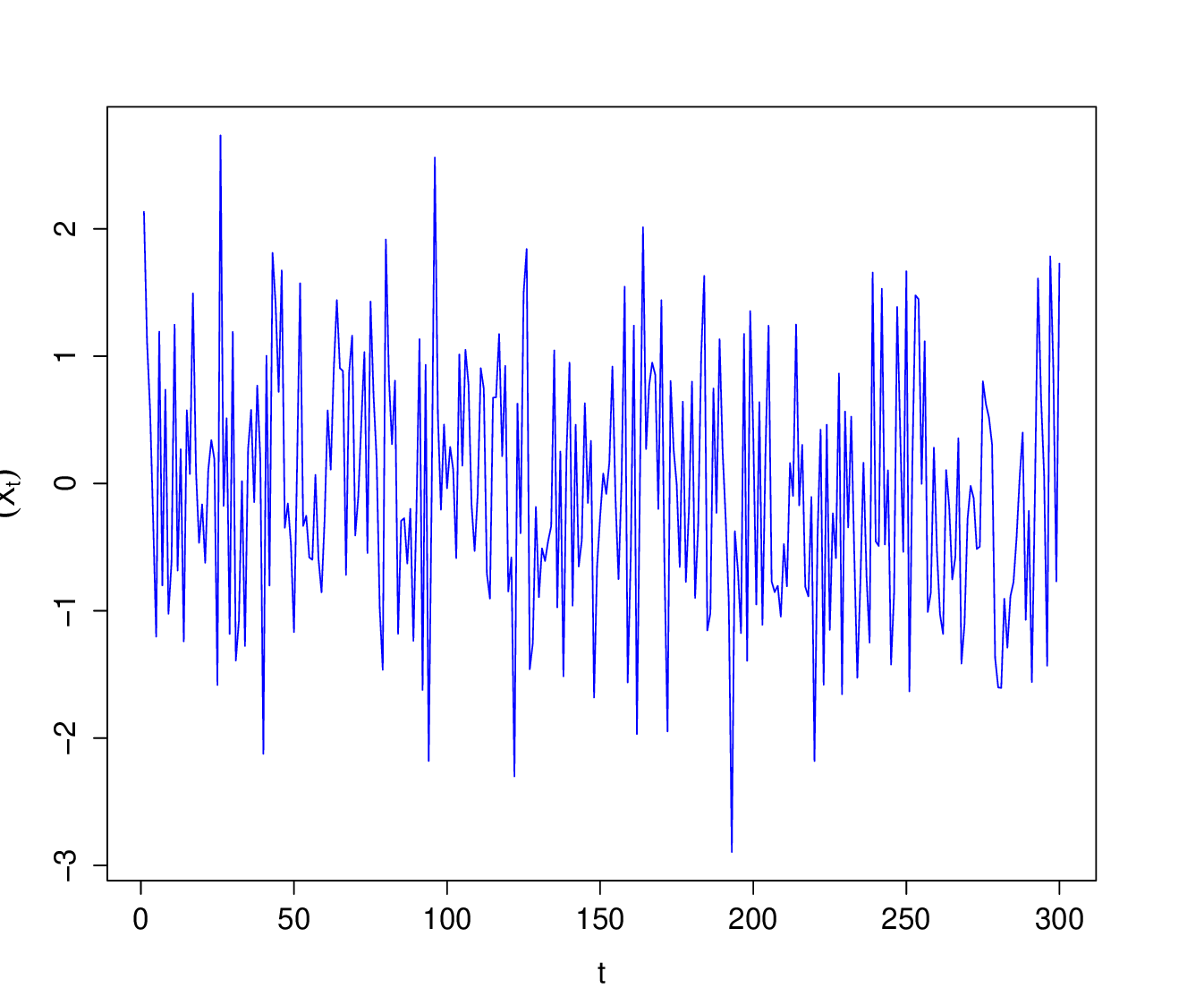}\! \includegraphics[width=4.9cm]{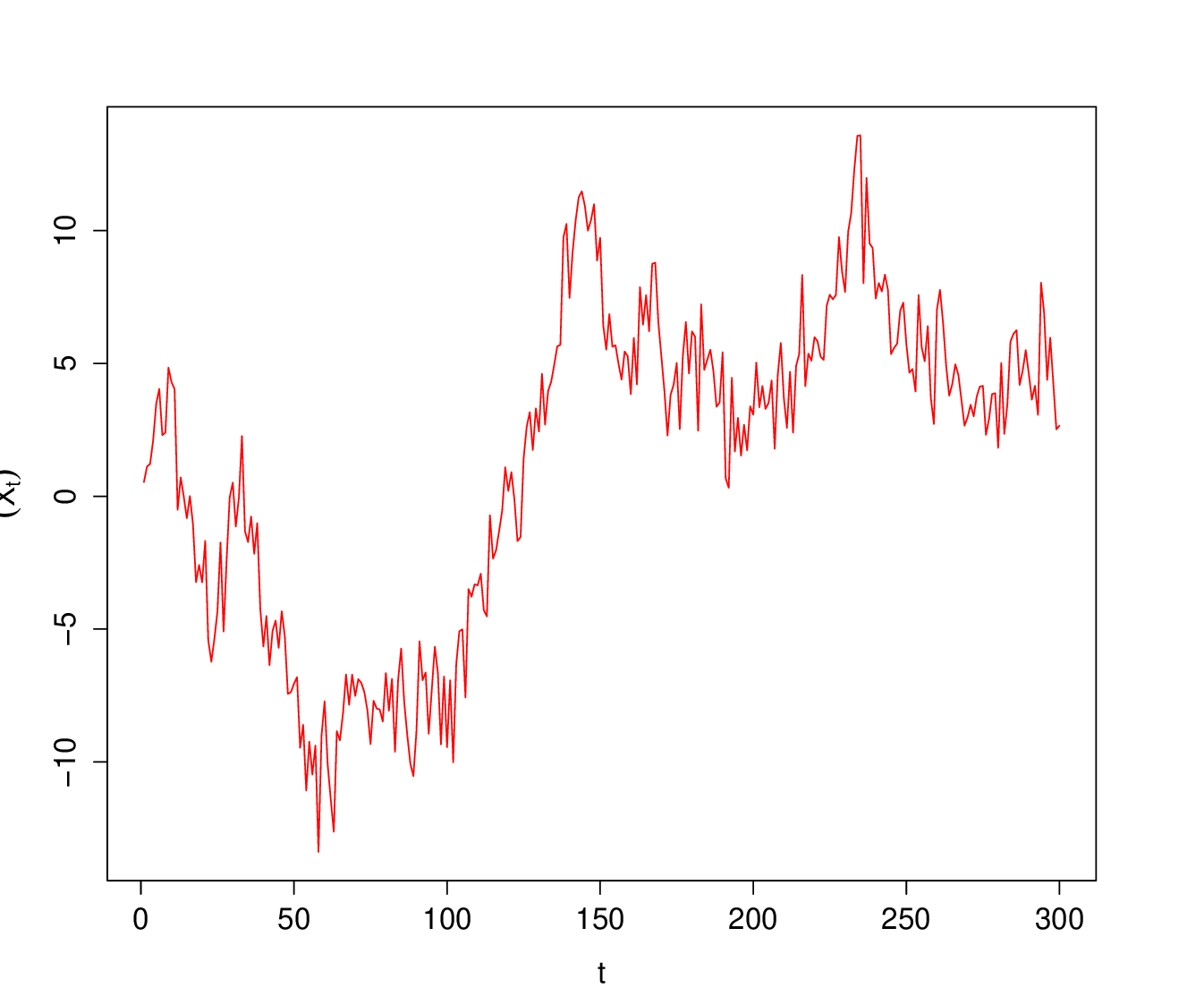}\! \includegraphics[width=4.9cm]{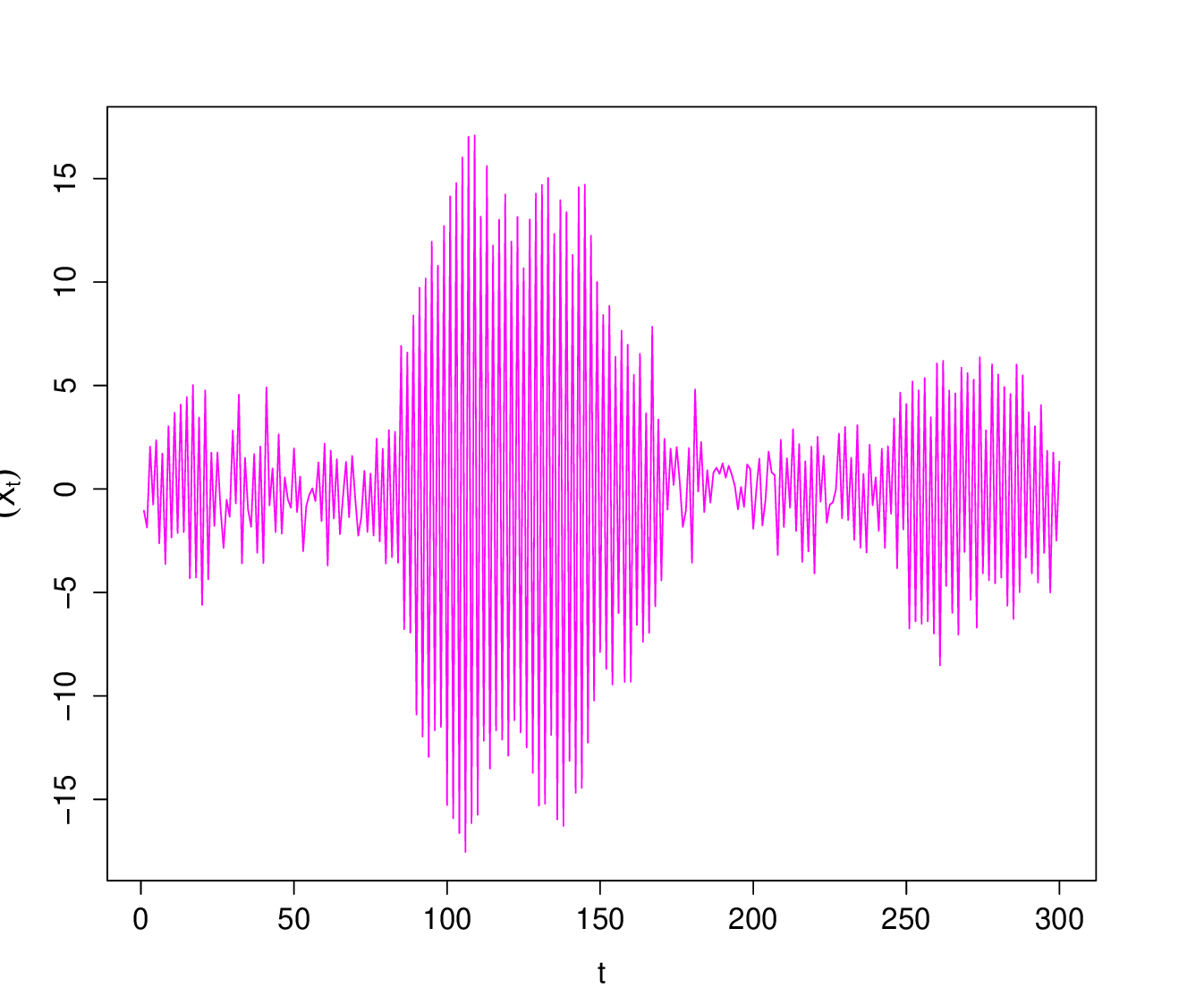}
\caption{\small Example of simulations under $\cH_0$ (left), under $\cH_1^{+}$ (middle) and under $\cH_1^{-}$ (right), for $T=300$, $p=0$, $\kappa=0$ and standard Gaussian white noises.}
\label{FigExSimK0}
\end{figure}

\begin{figure}[h!]
\centering
\includegraphics[width=4.9cm]{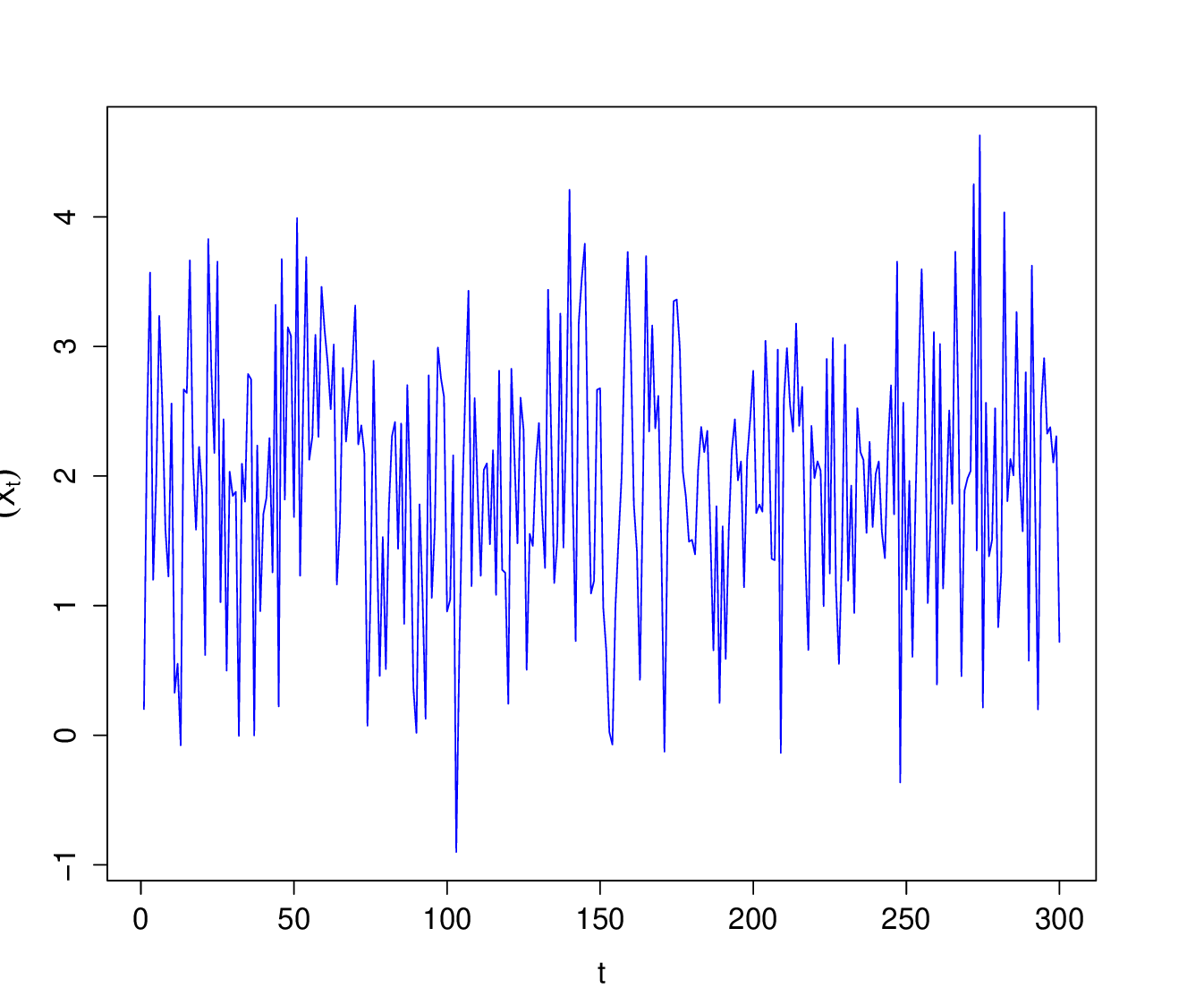}\! \includegraphics[width=4.9cm]{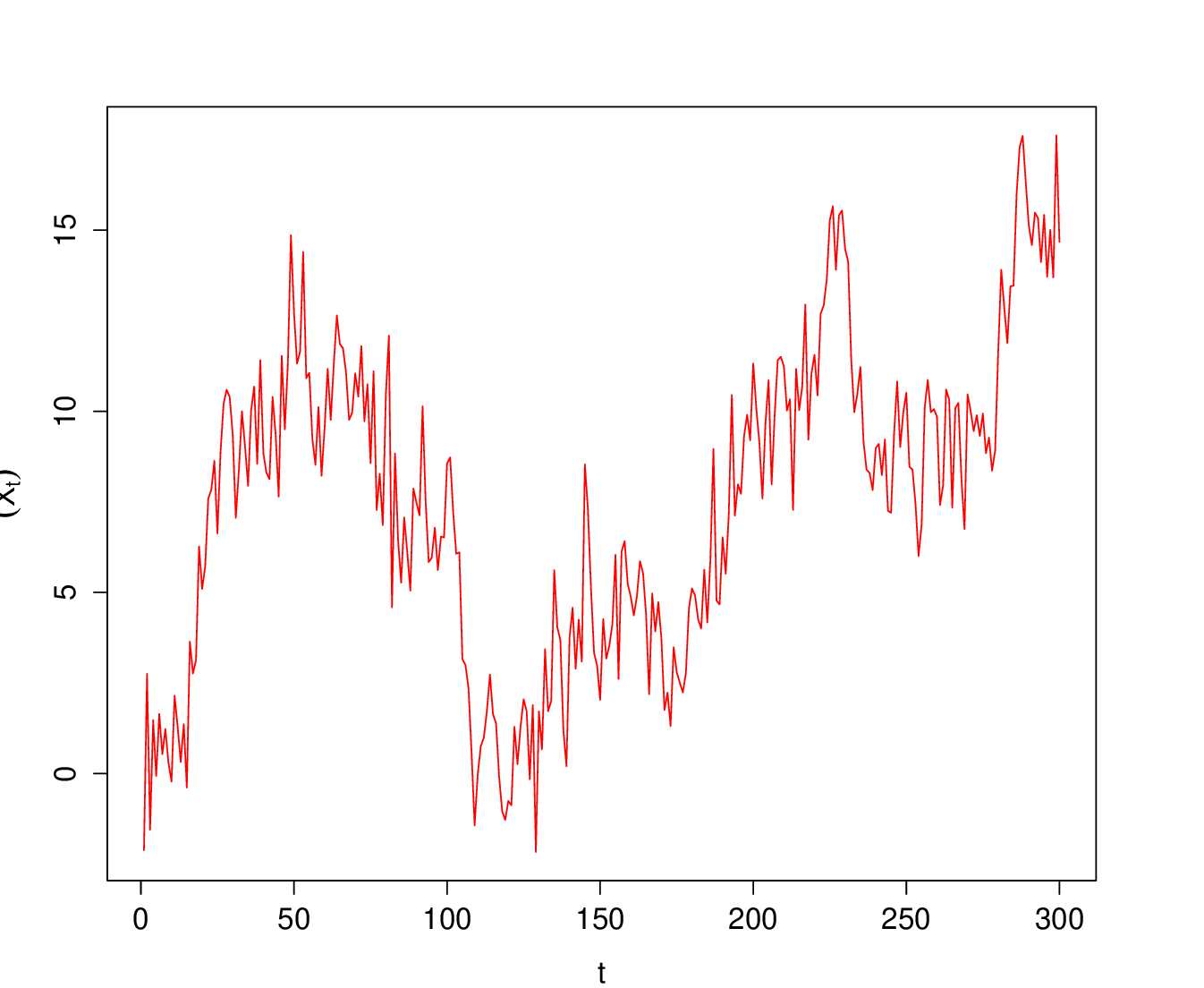}\! \includegraphics[width=4.9cm]{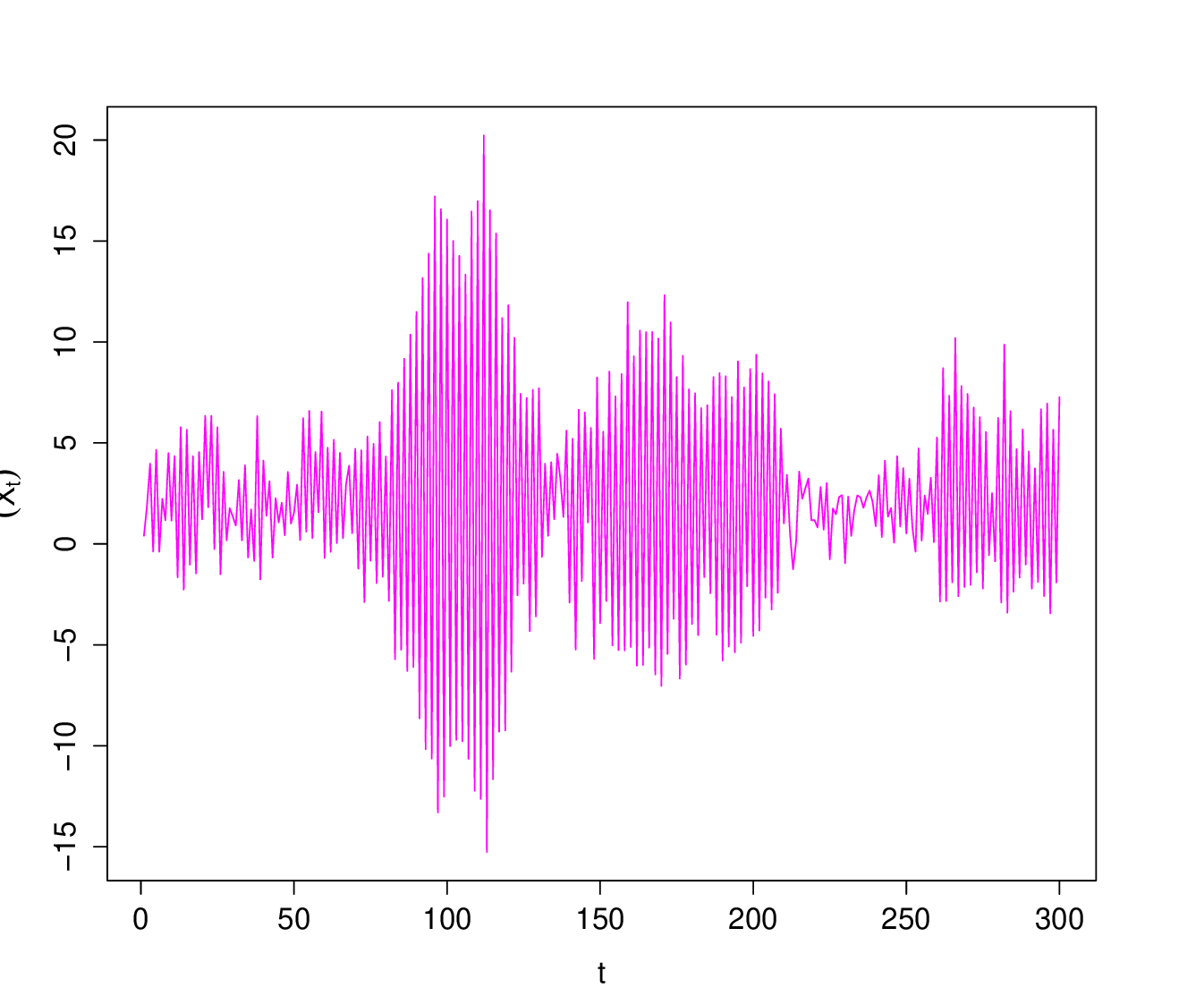}
\caption{\small Example of simulations under $\cH_0$ (left), under $\cH_1^{+}$ (middle) and under $\cH_1^{-}$ (right), for $T=300$, $p=0$, $\kappa \neq 0$, $r=0$ (with $a_0=2$) and standard Gaussian white noises.}
\label{FigExSimR0}
\end{figure}

\begin{figure}[h!]
\centering
\includegraphics[width=4.9cm]{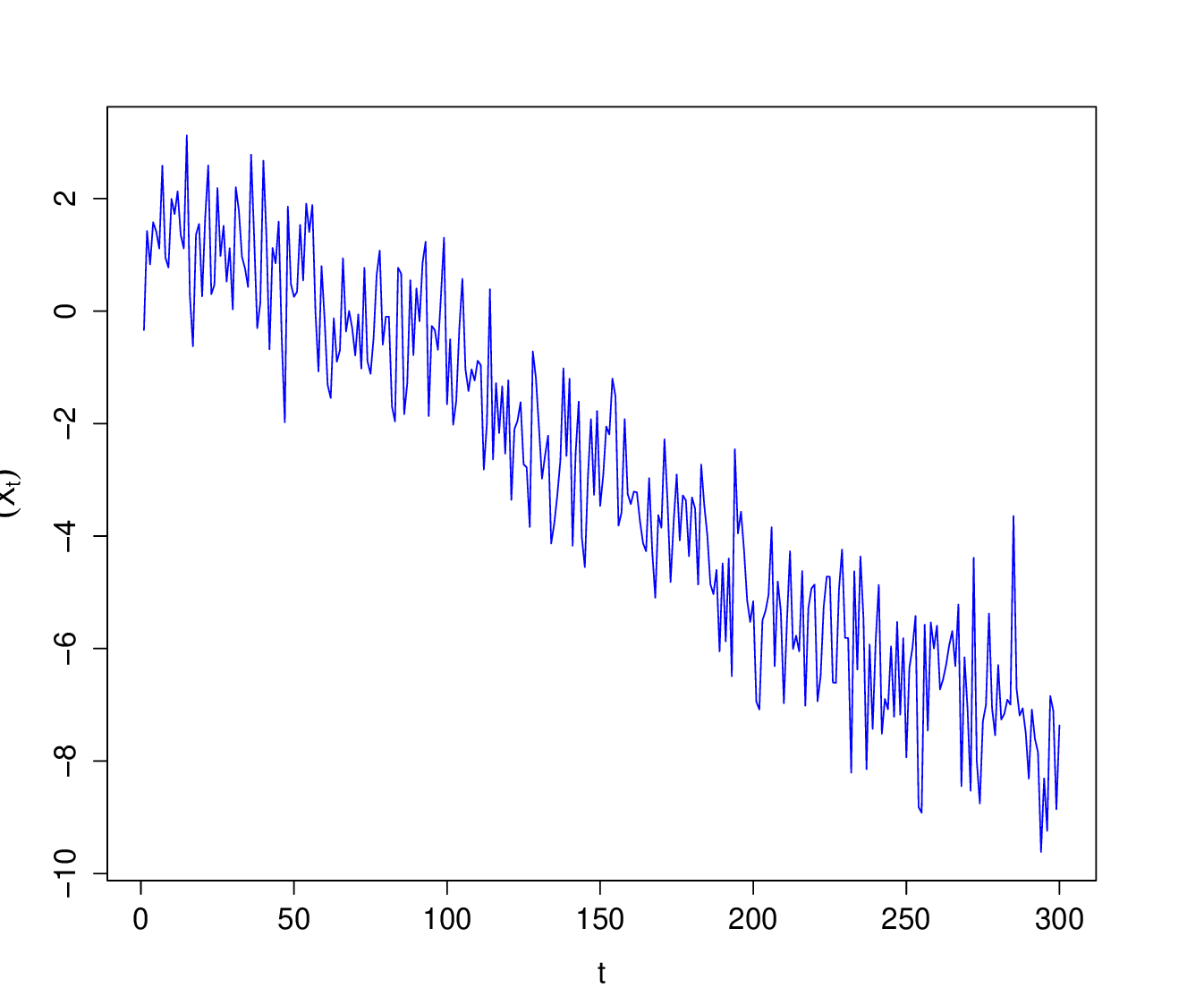}\! \includegraphics[width=4.9cm]{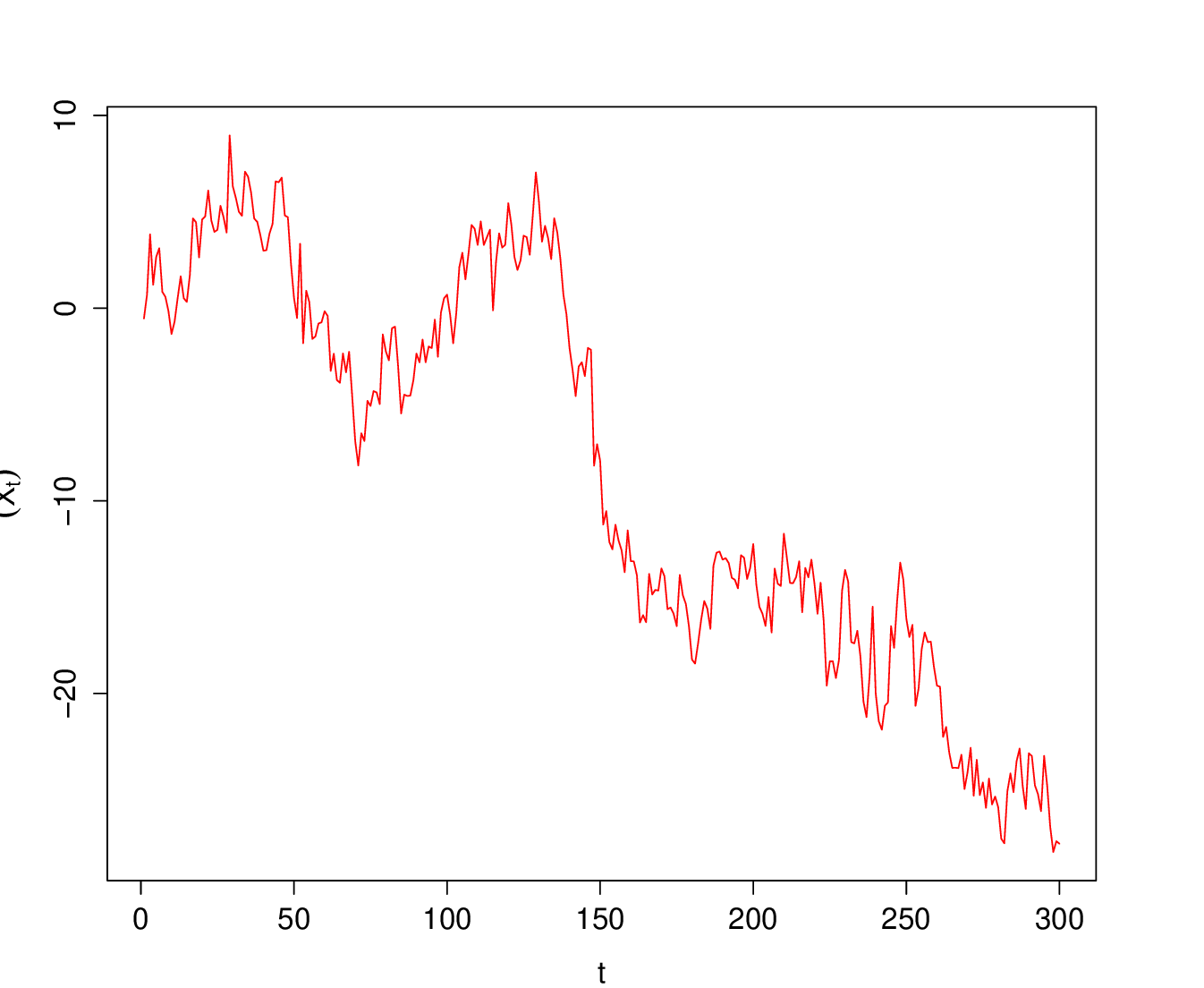}\! \includegraphics[width=4.9cm]{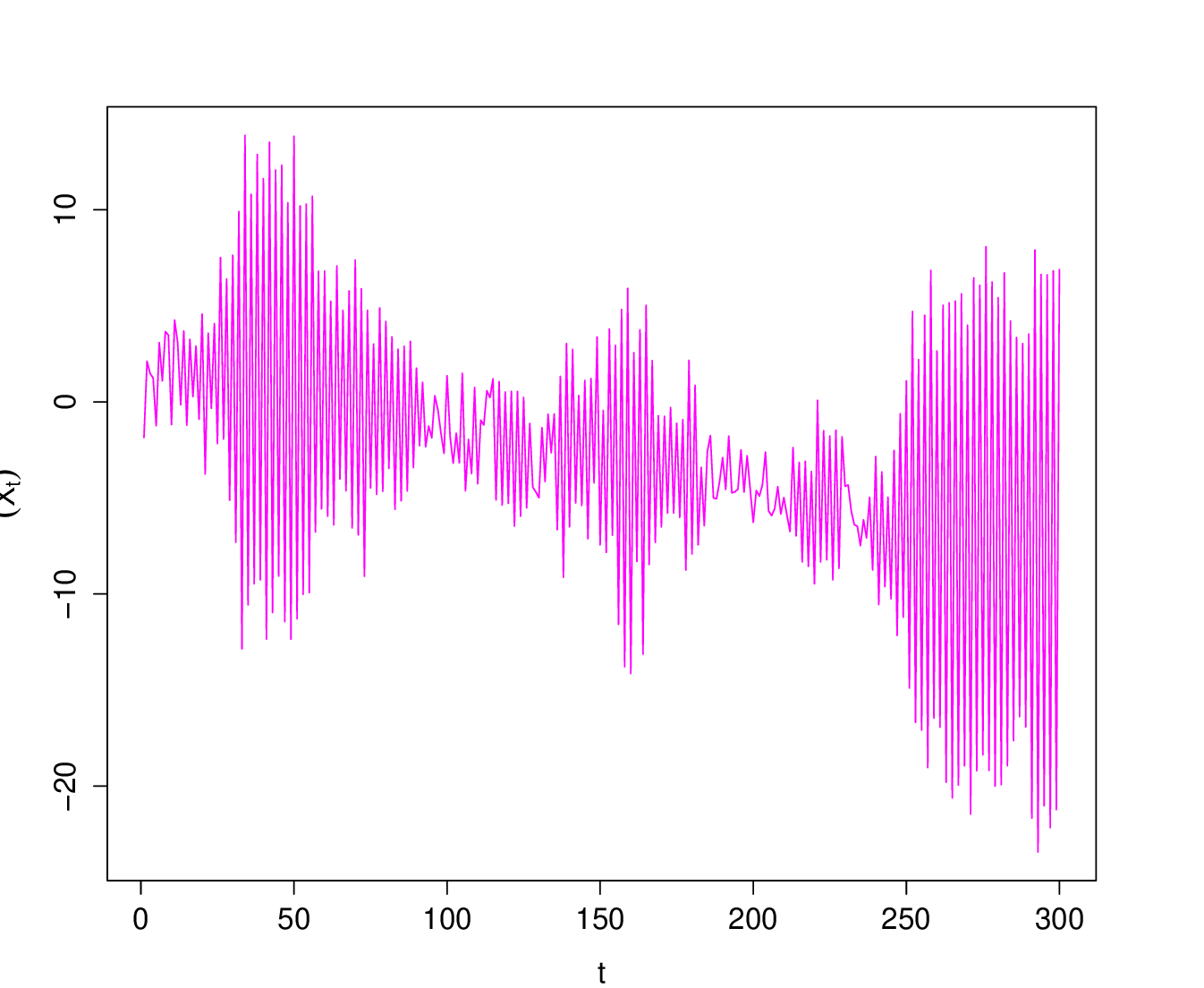}
\caption{\small Example of simulations under $\cH_0$ (left), under $\cH_1^{+}$ (middle) and under $\cH_1^{-}$ (right), for $T=300$, $p=0$, $\kappa \neq 0$, $r=1$ (with $a_0=2$ and $a_1=-10$) and standard Gaussian white noises.}
\label{FigExSimR1}
\end{figure}

\begin{figure}[h!]
\centering
\includegraphics[width=4.9cm]{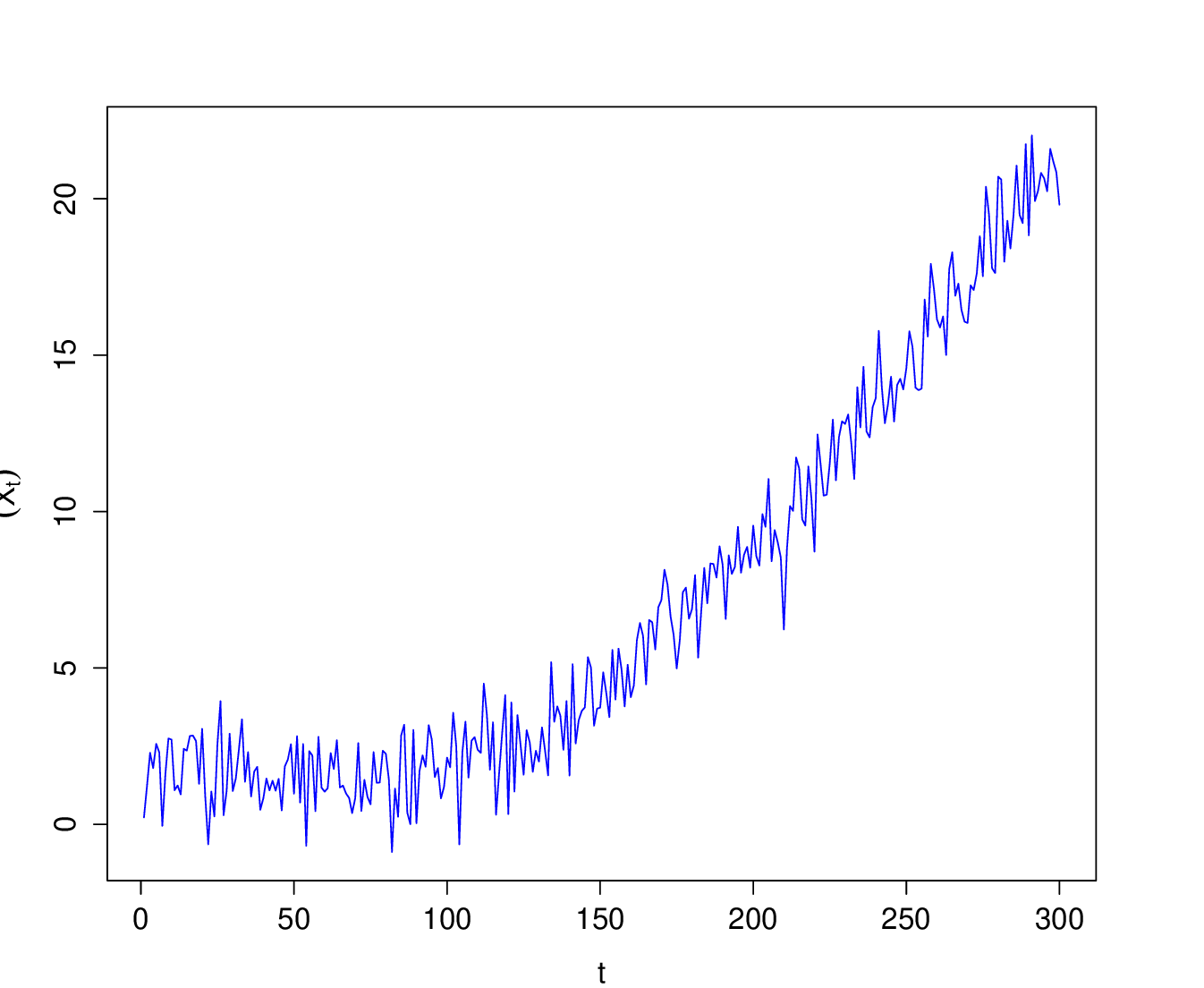}\! \includegraphics[width=4.9cm]{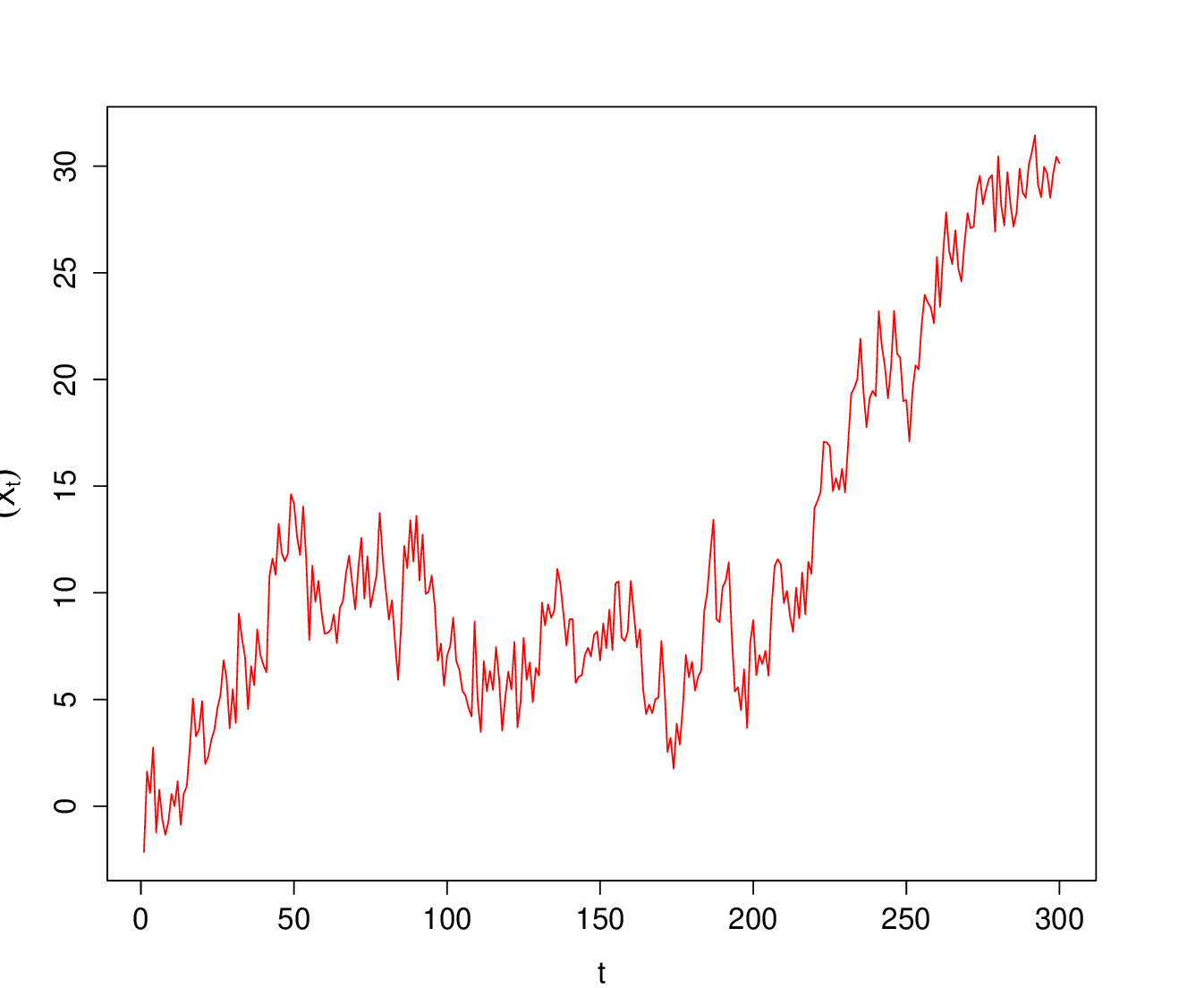}\! \includegraphics[width=4.9cm]{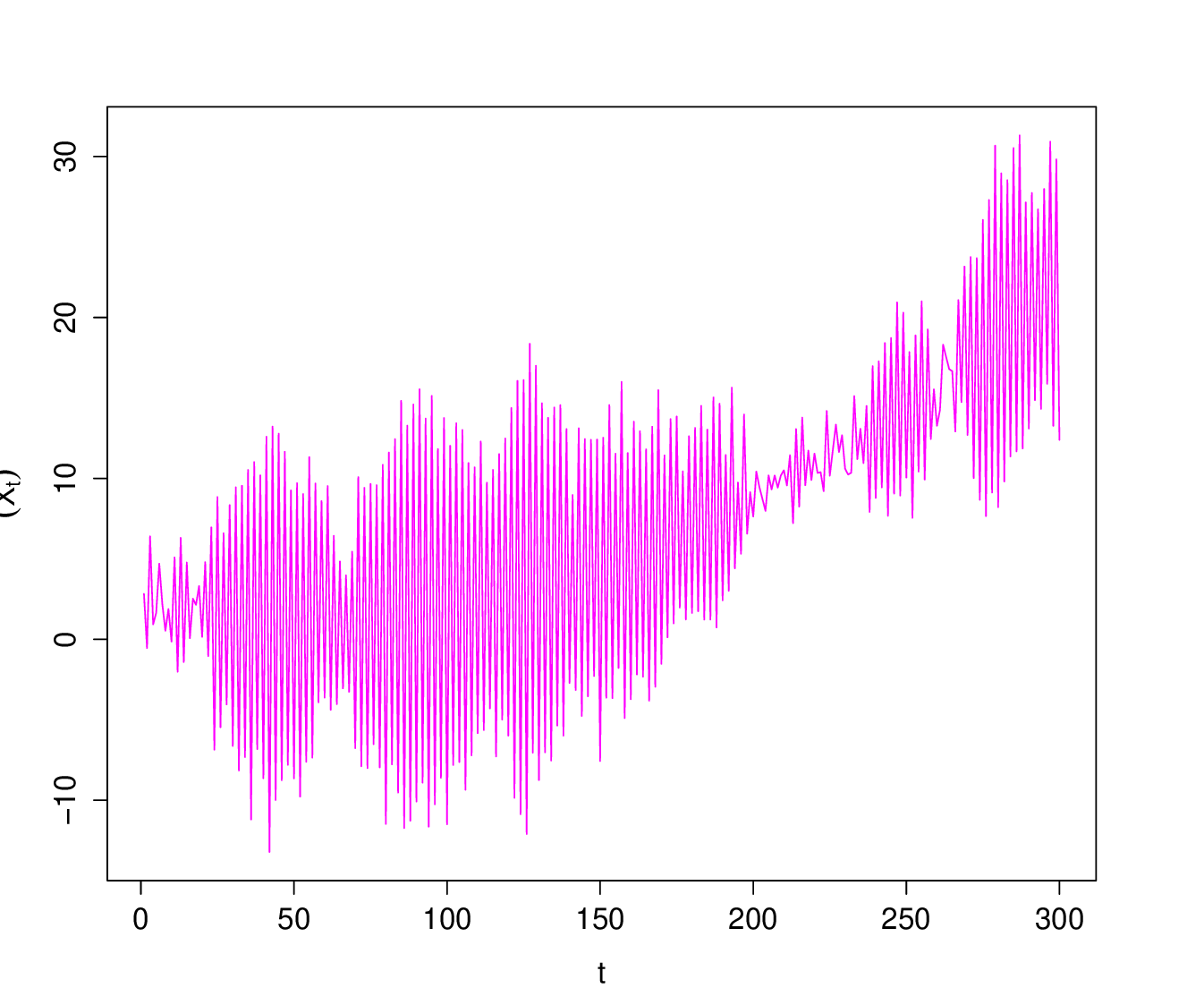}
\caption{\small Example of simulations under $\cH_0$ (left), under $\cH_1^{+}$ (middle) and under $\cH_1^{-}$ (right), for $T=300$, $p=0$, $\kappa \neq 0$, $r=2$ (with $a_0=2$, $a_1=-10$ and $a_2 = 30$) and standard Gaussian white noises.}
\label{FigExSimR2}
\end{figure}

\medskip

The first observation is that, due to the alternation generated by $\rho=-1$, it seems quite intuitive to choose between $\cH_1^{+}$ and $\cH_1^{-}$ to conduct the test. Besides, it is perceptible on the simulations that heteroscedasticity is manifest. Such high-frequency signals (under $\cH_1^{-}$) are quite unusual in the econometric field, and yet it remains a nonstationary eventuality that a consistent test needs to handle. In the particular case where $p=0$, $\kappa=0$ and where $(\veps_{t})$ and $(\eta_{t})$ are standard Gaussian white noises, we have conducted $N=10000$ simulations, each time testing for stationarity using the KPSS and the LMC procedures. We have obtained the following results (Table \ref{TabSim}). On the one hand, we observe that the size of each test is appropriate, since the procedures have been conducted with a significance level $\alpha=0.05$. One also observes that each test is consistent under $\cH_1^{+}$ but, as one can notice on Table \ref{TabSim} they are mislead under $\cH_1^{-}$ and do not detect this kind of nonstationarity.

\begin{table}[h!]
\begin{center}
\begin{tabular}{|c|c|c|}
\cline{2-3}
\multicolumn{1}{c|}{} & KPSS & LMC \\
\hline
Under $\cH_0$ & \small{0.051} & \small{0.051} \\
\hline
Under $\cH_1^{+}$ & \small{0.989} & \small{0.998} \\
\hline
Under $\cH_1^{-}$ & \small{0.043} & \small{0.010} \\
\hline
\end{tabular}
\bigskip
\end{center}
\caption{\small Frequency of rejection of the null hypothesis of stationarity on the basis of $N=10000$ simulations, using the KPSS and LMC procedures.}
\label{TabSim}
\end{table}

\medskip

This phenomenon is a direct consequence of Theorem \ref{ThmStatH1}, in which we have proved that $\whk$ converges to zero when the unit root of the integrated process is located at $-1$. To correct this misuse, we suggest to modify the rejecting rules of the usual procedures depending on whether the alternative is $\cH_1^{+}$ or $\cH_1^{-}$. Let $k_{r,\, \alpha}$ be the $\alpha$--quantile of the limiting distribution of Theorem \ref{ThmStatH0} for a given $r$, with the convention that $k_{r,\, \alpha} = k_{\alpha}$ if $\kappa=0$. Then the corrected test takes the following form,
$$
C_{T} = \dI_{\{ \whk\, \in\, \cR_{\alpha} \}}  \hsp \text{with} \hsp \left\{
\begin{array}{ll}
\cR_{\alpha} = ~ ]\, k_{r,\, 1-\alpha},\, +\infty\, [ & ~ \text{for} ~ \cH_0 ~ \text{vs} ~ \cH_1^{+} \\
\cR_{\alpha} = [\, 0,\, k_{r,\, \alpha}\,[ & ~ \text{for} ~ \cH_0 ~ \text{vs} ~ \cH_1^{-}.
\end{array}
\right.
$$
Defined as above, the corrected test is exactly the LMC test for $r \leq 1$ and $\cH_1^{+}$, the generalization lies in $r \geq 2$ and the correction lies in the whole situations under $\cH_1^{-}$. In the particular case where $p=0$, it is even possible to build a two-sided test for stationarity,
$$
C_{T} = \dI_{\{ \whk\, \in\, \cR_{\alpha} \}}  \hsp \text{with} \hsp \cR_{\alpha} = [\, 0,\, k_{r,\, \frac{\alpha}{2}}\,[ ~ \cup ~ ]\, k_{r,\, 1-\frac{\alpha}{2}},\, +\infty\, [
$$
which is adapted to test for $\cH_0$ against $\cH_1 = \cH_1^{-} \cup \cH_1^{+}$. Figure \ref{FigTests} gives an overview of the corresponding rejection areas. However, it is crucial to note that for $p \neq 0$, it may be problematic to get a consistent estimation of $\theta$ since we cannot stationarize the process without any information on $\rho$. The two-sided procedure is therefore useful only for $p=0$, \textit{i.e.} in the KPSS framework.

%\newpage

\begin{figure}[h!]
\centering
\includegraphics[width=7.2cm]{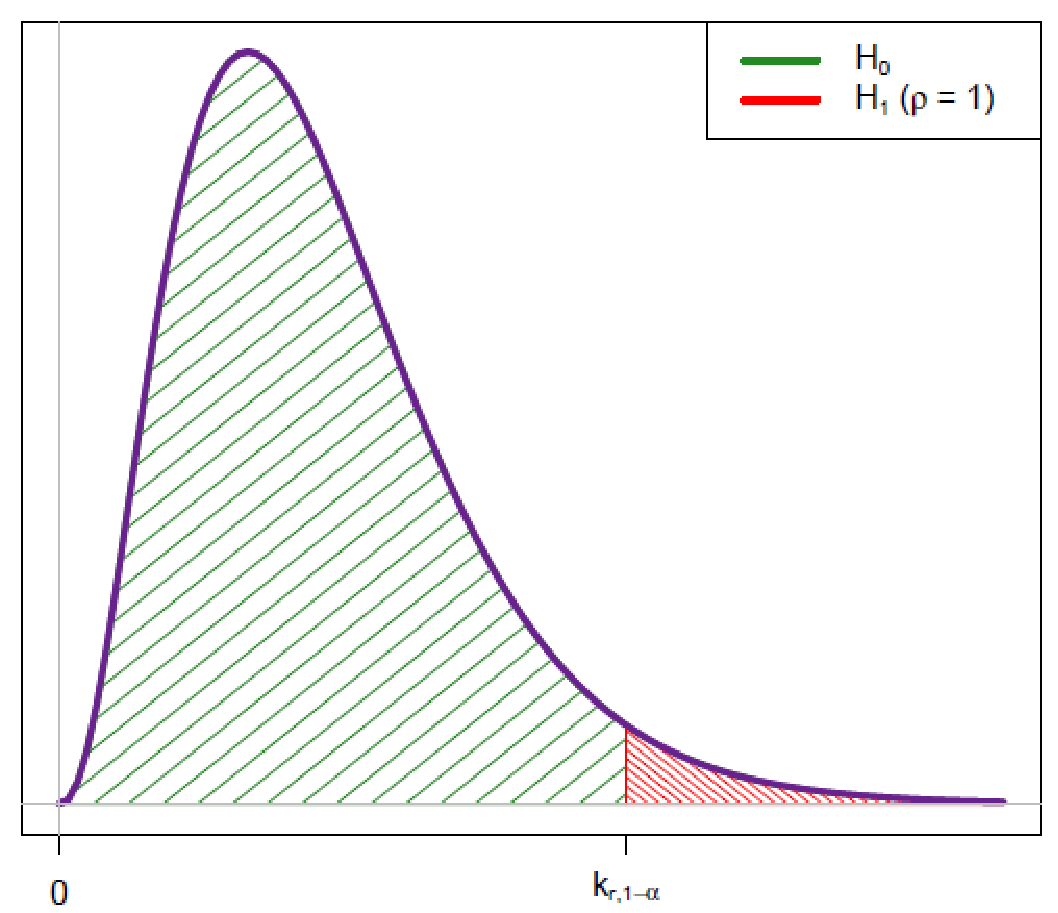} \includegraphics[width=7.2cm]{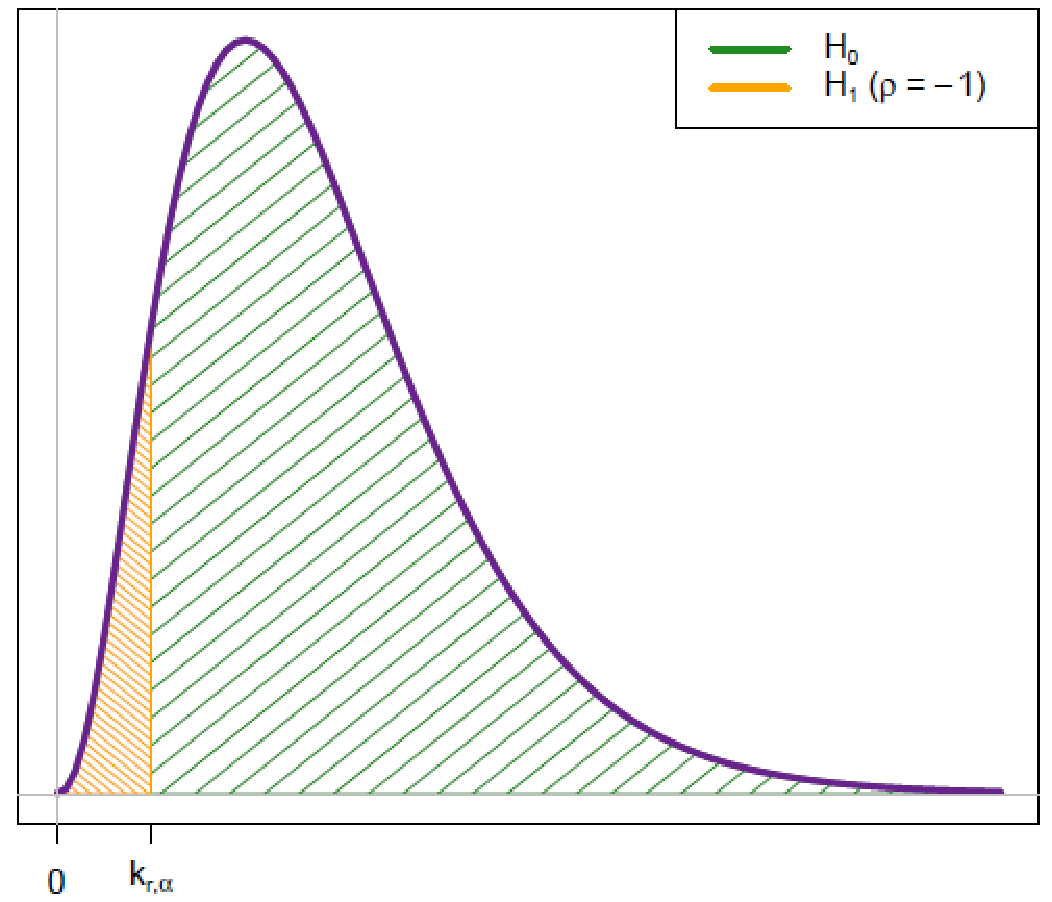}~\\~\\ \includegraphics[width=7.2cm]{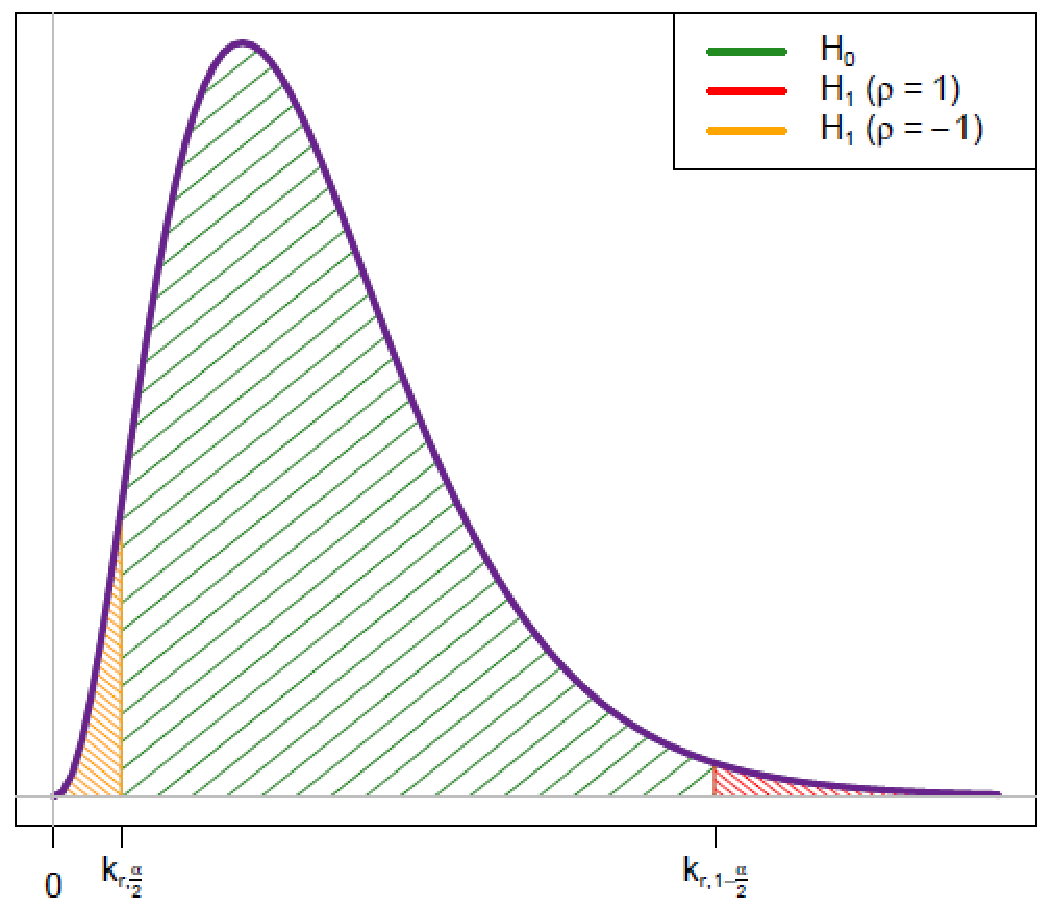}
\caption{\small Schematic representation of the rejection areas of $\cH_0$ to decide $\cH_1^{+}$ (top left), $\cH_1^{-}$ (top right) and $\cH_1^{-} \cup \cH_1^{+}$ (bottom), for a given significance level $\alpha$.}
\label{FigTests}
\end{figure}

\medskip

The application of the two-sided corrected test to the dataset used to fill Table \ref{TabSim} leads to 97.6 \% of rejection of $\cH_0$. With no doubt, this is a confirmation that $\cH_1^{-}$ is now correctly treated. The main corollary of the study is that our results should be rigorously driven to the KPSS procedure. Indeed, on the one hand, it is known that the LMC test suffers from size distortion for a stationary but strongly serially correlated process, as pointed out in \cite{CanerKilian01} or \cite{LanneSaikkonen03} among others, not forgetting that $p$ is always difficult to properly evaluate in an ARMA$(p,q)$ process. On the other hand, the corrected two-sided test could be conducted without choosing beforehand between $\cH_1^{+}$ and $\cH_1^{-}$ as the alternative. Such a test would be fully consistent for testing stationarity of ARMA processes, this is a trail for a future study.

%\newpage

%%%%%%%%%%%%%%%%%%%%%%%%%%%%%%%%%%%%%%%%%%%%%%%%%%%%%%%%%%%%%%%%%%%%%%%%%%%%%%%%

\section{Proof of the main results}
\label{SecProof}

%%%%%%%%%%%%%%%%%%%%%%%%%%%%%%%%%%%%%%%%%%%%%%%%%%%%%%%%%%%%%%%%%%%%%%%%%%%%%%%%

We are now going to prove our main results. We will consider in all the sequel the design matrix $A$ of order $(r+1) \times T$ defined as
\begin{equation}
\label{A}
A = \begin{pmatrix}
1 & 1 & \hdots & 1 & \hdots & 1 \\
1_{T} & 2_{T} & \hdots & k_{T} & \hdots & 1 \\
\vdots & \vdots & & \vdots & & \vdots \\
1_{T}^{r} & 2_{T}^{r} & \hdots & k_{T}^{r} & \hdots & 1
\end{pmatrix} \hsp \text{with} \hsp k_{T} = k/T.
\end{equation}
The Donsker's invariance principle and the Mann-Wald's continuity theorem being the cornerstone of all our reasonings, we found useful to remind them in this section.
\begin{thm}[Donsker]
\label{ThmDonsker}
Assume that $(Z_{T})$ is a sequence of independent and identically distributed random variables having mean 0 and finite variance $\sigma^2 > 0$. Let $S_0=0$ and $S_{T} = Z_1 + \hdots + Z_{T}$. For a given $0 < \tau \leq 1$, let also
\begin{equation*}
S_{T}^{(\tau)} = \frac{1}{\sigma \sqrt{T}} \left( S_{[T \tau]} + (T \tau - [T \tau]) Z_{[T \tau] + 1} \right).
\end{equation*}
Then, as $T$ goes to infinity, we have the weak convergence
\begin{equation*}
S_{T}^{(\tau)} \liml W(\tau)
\end{equation*}
where $W(t)$ is the standard Wiener process.
\end{thm}
\begin{thm}[Mann-Wald]
\label{ThmCMT}
Assume that $(Z_{T}, Z)$ is a sequence of random elements defined on a metric space $\cS$. Assume that the application $h : \cS \rightarrow \cS^{\prime}$, where $\cS^{\prime}$ is also a metric space, has a set of discontinuity points $\cD_{h}$ such that $\dP(Z \in \cD_{h}) = 0$. Then, as $T$ goes to infinity,
\begin{equation*}
Z_{T} \longrightarrow Z \hsp \Longrightarrow \hsp h(Z_{T}) \longrightarrow h(Z).
\end{equation*}
The implication holds for the convergence in distribution, the convergence in probability and the almost sure convergence.
\end{thm}
\begin{proof}
The Donsker's invariance principle is described and proved in Section 8 of \cite{Billingsley99}. The Mann-Wald's continuity theorem, usually called \textit{continuous mapping theorem}, is for example introduced in Theorem 2.7 of \cite{Billingsley99} and proved thereafter.
\end{proof}

In addition, we need to introduce an invariance principle for the residuals of the regression of a random sequence on a polynomial trend in the case where the disturbance has an integrated component. This is an extension of Theorem 1(d) of \cite{Stock99}. For $\kappa = 0$ but with a more general kind of perturbation, one can also find the foundations of this strategy in \cite{IbragimovPhillips08}.
\begin{lem}
\label{LemInvPRes}
Consider, for all $1 \leq t \leq T$, the model
\begin{equation*}
X_{t} = \alpha_0 + \alpha_1 t_{T} + \hdots + \alpha_{r}\, t_{T}^{r} + S_{t}^{(d)} + \veps_{t}
\end{equation*}
with $d \geq 1$ and $\kappa \neq 0$. Let $\wh{\alpha}_{T}$ be the least squares estimator of $\alpha$ and $(\whveps_{t})$ the estimated residual set. Then, we have the weak convergence
\begin{equation*}
\frac{\whveps_{[T \tau]}}{\sigma_{\eta} T^{d-1/2}} \liml W_{r,\, d-1}(\tau)
\end{equation*}
where where $(W_{r,\, d-1}(t),\, t \in [0,1])$ is the detrended Wiener process of order $r \times (d-1)$.
\end{lem}
\begin{proof}
Recall that $(S_{t}^{(d)})$ is a random walk of order $d$ generated by a white noise sequence $(\eta_{t})$ of variance $\sigma_{\eta}^2 > 0$, that we can define as
\begin{equation}
\label{ModRandWalkMulti}
\left\{
\begin{array}{l}
S_{t}^{(d)} = S_{t-1}^{(d)} + S_{t}^{(d-1)} \\
\vdots \\
S_{t}^{(2)} = S_{t-1}^{(2)} + S_{t}^{(1)} \\
S_{t}^{(1)} = S_{t-1}^{(1)} + \eta_{t}
\end{array}
\right.
\end{equation}
where we consider to lighten the calculations that $S_0^{(1)} = \hdots = S_0^{(d)} = 0$. The least squares estimator of $\alpha$ is given by
\begin{equation}
\label{OLS}
\wh{\alpha}_{T} = \left( \sum_{t=1}^{T} A_{t}\, A_{t}^{\prime} \right)^{\! -1} \sum_{t=1}^{T} A_{t} X_{t} = R_{T}^{-1} \sum_{t=1}^{T} A_{t} X_{t}
\end{equation}
where $A_{t}$ is the $t-$th column of $A$ given by \eqref{A}. It follows that
\begin{equation}
\label{DiffOLS}
\wh{\alpha}_{T} - \alpha = R_{T}^{-1} P_{T} \hsp \hsp \text{with} \hsp \hsp P_{T} = \sum_{t=1}^{T} A_{t} w_{t}
\end{equation}
in which we define the residual $w_{t} = S_{t}^{(d)} + \veps_{t}$. We start by establishing an invariance principle for $(w_{t})$. First, Theorem \ref{ThmDonsker} is sufficient to get
\begin{equation}
\label{InvPRes1}
\frac{S_{[T \tau]}^{(1)}}{\sigma_{\eta} \sqrt{T}} = \frac{1}{\sigma_{\eta} \sqrt{T}} \sum_{t=1}^{[T \tau]} \eta_{t} \liml W(\tau).
\end{equation}
By extension,
\begin{equation}
\label{InvPRes2}
\frac{S_{[T \tau]}^{(2)}}{\sigma_{\eta} T^{3/2}} = \frac{1}{\sigma_{\eta} T^{3/2}} \sum_{t=1}^{[T \tau]} S_{t}^{(1)} = \sum_{t=1}^{[T \tau]} \int_{\frac{t}{T}}^{\frac{t+1}{T}} \frac{S_{[T s]}^{(1)}}{\sigma_{\eta} T^{1/2}} \dd s \liml \int_0^{\tau} W(s) \dd s \equiv W^{(1)}(\tau)
\end{equation}
from Theorem \ref{ThmCMT}. Iterating the process, we obtain, for $d \geq 2$,
\begin{equation}
\label{InvP_Sd}
\frac{S_{[T \tau]}^{(d)}}{\sigma_{\eta} T^{d-1/2}} \liml \int_0^{\tau} \int_0^{s_{1}} \hdots \int_0^{s_{d-2}} W(s_{d-1}) \dd s_{d-1} \hdots \dd s_{1} \equiv W^{(d-1)}(\tau).
\end{equation}
Since $\veps_{[T \tau]} = o(T^{d-1/2})$ a.s. from the strong law of large numbers, it follows that $(w_{t})$ also satisfies the invariance principle given by \eqref{InvP_Sd}, for all $d \geq 1$. For $d=1$, one can identify the limiting distribution in \eqref{InvP_Sd} and $\sigma_{\eta}$ to $W$ and $\sqrt{\omega}$ in Assumption 1(a) of \cite{Stock99}. In addition, the $k-$th line of $P_{T}$ given in \eqref{DiffOLS} is
\begin{equation}
\label{Pk}
P_{k,\, T} = \sum_{t=1}^{T} t_{T}^{k-1} w_{t} = \frac{1}{T^{k-1}} \sum_{t=1}^{T} t^{k-1} w_{t}.
\end{equation}
We are now going to study the rate of convergence of $P_{k,\, T}$. For all $1 \leq i \leq d$, denote $\delta_{k}(i) = i+k-1/2$. We can use \eqref{InvP_Sd} to get
\begin{equation}
\label{InvP_SumTkW}
\frac{1}{\sigma_{\eta} T^{\delta_{k}(d)}} \sum_{t=1}^{[T \tau]} t^{k-1} w_{t} = \sum_{t=1}^{[T \tau]} \int_{\frac{t}{T}}^{\frac{t+1}{T}} \frac{[T s]^{k-1} w_{[T s]}}{\sigma_{\eta} T^{k-1} T^{\delta_{0}(d)}} \dd s \liml \int_0^{\tau} s^{k-1} W^{(d-1)}(s) \dd s.
\end{equation}
By combining \eqref{Pk} and \eqref{InvP_SumTkW}, we find that, for all $d \geq 1$,
\begin{equation}
\label{InvP_P}
\frac{P_{[T \tau]}}{\sigma_{\eta} T^{d+1/2}} \liml P_{d}(\tau)
\end{equation}
where the limiting distribution is given in \eqref{Pd}. Moreover, by a direct calculation,
\begin{equation}
\label{LimR}
\limT \frac{R_{T}}{T} = M \hsp \text{and} \hsp \limT T R_{T}^{-1} = M^{-1}
\end{equation}
where $R_{T}$ is given in \eqref{OLS} and the nonsingular matrix $M$ satisfies $M_{i j} = 1/(i+j-1)$ for all $1 \leq i,j \leq r+1$. It follows from \eqref{DiffOLS}, \eqref{InvP_P} and \eqref{LimR} that
\begin{equation}
\label{Tlc_OLS}
\frac{\wh{\alpha}_{T} - \alpha}{\sigma_{\eta} T^{d - 1/2}} \liml M^{-1} P_{d}(1).
\end{equation}
It only remains to notice that
\begin{equation}
\label{DecompResOLS}
\frac{\whveps_{[T \tau]}}{T^{d-1/2}} = \frac{w_{[T \tau]}}{T^{d-1/2}} - \frac{\big( \wh{\alpha}_{T} - \alpha \big)^{\prime} A_{[T \tau]}}{T^{d-1/2}}
\end{equation}
and to combine \eqref{InvP_Sd} and \eqref{Tlc_OLS} to conclude that, for $d \geq 1$,
\begin{equation*}
\frac{\whveps_{[T \tau]}}{\sigma_{\eta} T^{d-1/2}} \liml W^{(d-1)}(\tau) - P_{d}^{\, \prime}(1) M^{-1} \Lambda(\tau) \equiv W_{r,\, d-1}(\tau)
\end{equation*}
from Theorem \ref{ThmCMT}, where $\Lambda(\tau) = \begin{pmatrix} 1 & \tau & \hdots & \tau^{r} \end{pmatrix}^{\prime}$ is the limiting value of $A_{[T \tau]}$. For $d=1$, the latter convergence is given in Theorem 1(d) of \cite{Stock99}. This achieves the proof of Lemma \ref{LemInvPRes}.
\end{proof}

\medskip

\noindent \textbf{Proof of Theorem \ref{ThmStatH0}.} Denote by $P = A^{\prime} (A A^{\prime})^{-1} A$ the projection matrix and by $I_{T}$ the identity matrix of order $T$. We start by expressing $(\whveps_{t})$ in terms of $(\veps_{t})$ to establish an invariance principle such as Theorem \ref{ThmDonsker} on $(S_{t})$ given by \eqref{SumEstRes}. We first consider the general case where $\kappa \neq 0$. From \eqref{Ystar} and \eqref{EstRes}, since $\wh{\alpha}_{T}$ is the least squares estimator of $\alpha$, a direct calculation shows that, for all $1 \leq t \leq T$,
\begin{equation}
\label{ResSingleH0}
\whveps_{t} = \check{X}_{t} - \wh{\alpha}_0 - \wh{\alpha}_1 t_{T} - \hdots - \wh{\alpha}_{r}\, t_{T}^{r} = \sum_{i=1}^{p} (\theta_{i} - \check{\theta}_{i})\, u_{i,\,t} + u_{t}
\end{equation}
where $u_{t}$ is the $t-$th component of $(I_{T}-P) \veps$, and, for $1 \leq i \leq p$, $u_{i,\,t}$ is the $t-$th component of $(I_{T}-P) X_{-i}$ with $X_{-i}^{\, \prime} = \begin{pmatrix} X_{1-i} & \hdots & X_{T-i} \end{pmatrix}$. From Theorem 1 of \cite{MacNeill78}, we have the weak convergence
\begin{equation}
\label{InvP_ResH0}
\frac{1}{\sigma_{\veps} \sqrt{T}} \sum_{t=1}^{[T \tau]} u_{t} \liml B_{r}(\tau).
\end{equation}
In addition, for any $1 \leq i \leq p$ and since $\Theta$ is causal, equation \eqref{ModSingle} leads to
\begin{equation}
\label{Yi}
X_{t-i} = \Theta^{-1}(L)(\alpha_0 + \alpha_1 (t-i)_{T} + \hdots + \alpha_{r} (t-i)_{T}^{r}) + \mu_{t-i}
\end{equation}
where $(t-i)_{T} = (t-i)/T$ and $\Theta(L) \mu_{t-i} = \veps_{t-i}$. The coefficients of the deterministic trend are identifiable \textit{via} a tedious but straightforward calculation. It follows from \eqref{Yi} that $(\mu_{t})$ is a stable stationary AR$(p)$ process which also satisfies an invariance principle, as it is stipulated for example in Theorem 1 of \cite{DedeckerRio00}. If we define the so-called long-run variance as
\begin{equation*}
\sigma^2_{\mu} = \dE\, [\mu_0^2] + 2 \sum_{k=1}^{\infty} \dE\, [\mu_0 \mu_{k}]
\end{equation*}
which is finite for a stable AR process (see Chapter 3 of \cite{BrockwellDavis96}), then, for all $1 \leq i \leq p$,
\begin{equation}
\label{InvP_ARResH0}
\frac{1}{\sigma_{\mu} \sqrt{T}} \sum_{t=1}^{[T \tau]} u_{i,\,t} \liml B_{r}(\tau),
\end{equation}
by using again Theorem 1 of \cite{DedeckerRio00}. Convergence \eqref{InvP_ARResH0} and the consistency of $\check{\theta}_{T}$ imply that
\begin{equation}
\label{CvgResARH0}
\frac{1}{\sigma_{\mu} \sqrt{T}} \sum_{i=1}^{p} (\theta_{i} - \check{\theta}_{i}) \sum_{t=1}^{[T \tau]} u_{i,\,t} \limp 0.
\end{equation}
Noticing that $(S_{t})$ in \eqref{SumEstRes} is the partial sum process of $(\whveps_{t})$, it follows that
\begin{equation}
\label{InvP_SH0}
\frac{S_{[T \tau]}}{\sigma_{\veps} \sqrt{T}} \liml B_{r}(\tau).
\end{equation}
In addition, it is not hard to see that
\begin{equation*}
\limT \frac{1}{T} \sum_{t=1}^{T} u_{t}^2 = \sigma^2_{\veps} \cvgps
\end{equation*}
since $(u_{t})$ can be seen as the residual of the regression of $(\veps_{t})$ on a polynomial time trend with zero coefficients. The same kind of convergence can be reached for $(u_{i,\,t})$ following a similar methodology as in \cite{PhillipsPerron88}, since $(u_{i,\,t})$ can be seen as the residual of the regression of a weakly stationary process $(\mu_{t})$ on a polynomial time trend also with zero coefficients. Hence, by the Cauchy-Schwarz's inequality,
\begin{equation}
\label{CvgEstVarH0}
\limT \frac{Q_{T}}{T} = \sigma^2_{\veps} \cvgps
\end{equation}
where the process $(Q_{t})$ is given by \eqref{SumEstRes}. Finally,
\begin{equation*}
\frac{1}{\sigma^2_{\veps}\, T^2} \sum_{t=1}^{[T \tau]} S_{t}^{\, 2} = \frac{1}{T}\, \sum_{t=1}^{[T \tau]} \left( \frac{S_{t}}{\sigma_{\veps} \sqrt{T}} \right)^{\! 2} = \sum_{t=1}^{[T \tau]} \int_{\frac{t}{T}}^{\frac{t+1}{T}} \left(  \frac{S_{[T s]}}{\sigma_{\veps} \sqrt{T}} \right)^{\! 2} \dd s \liml \int_0^{\tau} B_{r}^{\, 2} (s) \dd s
\end{equation*}
by application of Theorem \ref{ThmCMT}. This achieves the proof of Theorem \ref{ThmStatH0}, using \eqref{InvP_SH0}, \eqref{CvgEstVarH0}, Slutsky's lemma and taking $\tau = 1$, in the case where there is a polynomial trend. On the other hand, for $\kappa = 0$, $P$ is the zero matrix and we merely have $u_{t} = \veps_{t}$ and $u_{i,\,t} = X_{-i}$ in \eqref{ResSingleH0}, for all $1 \leq t \leq T$ and $1 \leq i \leq p$. Then, convergence \eqref{CvgEstVarH0} follows from the strong law of large numbers and, by Theorem \ref{ThmDonsker}, the invariance principle \eqref{InvP_SH0} becomes
\begin{equation}
\label{InvP_SH0NT}
\frac{S_{[T \tau]}}{\sigma_{\veps} \sqrt{T}} \liml W(\tau).
\end{equation}
The end of the proof follows the same reasoning as above. \hfill
$\mathbin{\vbox{\hrule\hbox{\vrule height1ex \kern.5em\vrule height1ex}\hrule}}$

\medskip

\noindent \textbf{Proof of Theorem \ref{ThmStatH1}.} We now suppose that $\sigma^2_{\eta} > 0$, implying that the process has a stochastic nonstationarity generated by the random walk $(S_{t}^{\eta})$ given by \eqref{RWSingle}. We first consider the general case $\kappa \neq 0$. In the same way as for \eqref{ResSingleH0}, we obtain
\begin{equation}
\label{ResSingleH1}
\whveps_{t} = \check{X}_{t} - \wh{\alpha}_0 - \wh{\alpha}_1 t_{T} - \hdots - \wh{\alpha}_{r}\, t_{T}^{r} = \sum_{i=1}^{p} (\theta_{i} - \check{\theta}_{i})\, u_{i,\,t} + u_{\eta,\,t}
\end{equation}
where $u_{\eta,\,t}$ is the $t-$th component of $(I_{T}-P) (S^{\eta} + \veps)$. In addition, for all $1 \leq i \leq p$, $u_{i,\,t}$ is the $t-$th component of $(I_{T}-P) X_{-i}$ and $X_{-i}$ is given, for all $1 \leq t \leq T$, by
\begin{equation}
\label{YiH1}
X_{t-i} = \Theta^{-1}(L)(\alpha_0 + \alpha_1 (t-i)_{T} + \hdots + \alpha_{r} (t-i)_{T}^{r}) + T^{\eta}_{t-i}
\end{equation}
and $\Theta(L) T^{\eta}_{t-i} = S^{\eta}_{t-i} + \veps_{t-i}$, with the notations of \eqref{Yi}. Hence, $((I - \rho L) T^{\eta}_{t-i})$ is second-order equivalent in moments to a stationary ARMA$(p,1)$ process. From Theorem 1 of \cite{DedeckerRio00}, it satisfies an invariance principle in which its long-run variance is involved, and the rate is $\sqrt{T}$. Then, by Theorem \ref{ThmCMT} and standard calculations, one can see that $(u_{i,\,t})$ behaves like $(u_{\eta,\,t})$ since all invariance principles on $(u_{\eta,\,t})$ can also be established on $(u_{i,\,t})$. However as $\check{\theta}_{T}$ is consistent, it appears that all asymptotic results will only be driven by $(u_{\eta,\,t})$, $(u_{\eta,\,t}^{2})$ and their partial sum processes. First, by Theorem \ref{ThmDonsker} in the case where $\rho=1$, we have already seen in \eqref{InvPRes1} that we have the invariance principle
\begin{equation}
\label{InvP_SetaH1}
\frac{S_{[T \tau]}^{\eta}}{\sigma_{\eta} \sqrt{T}} \liml W(\tau).
\end{equation}
For $\rho=-1$, one cannot directly apply Theorem \ref{ThmDonsker} since $(S^{\eta}_{t})$ is not built from identically distributed random variables. However, convergence \eqref{InvP_SetaH1} still holds by using for example Theorem 1 of \cite{DedeckerRio00}. Depending on the value of $\rho$, the end of the proof is totally different. On the one hand, for $\rho=1$, from Lemma \ref{LemInvPRes} with $d=1$, we have the weak convergence
\begin{equation}
\label{InvP_ResSetaH1}
\frac{u_{\eta,\, [T \tau]}}{\sigma_{\eta} \sqrt{T}} \liml W_{r,\, 0}(\tau).
\end{equation}
It follows that
\begin{equation}
\label{InvP_SumResSetaH1}
\frac{1}{\sigma_{\eta} T^{3/2}} \sum_{t=1}^{[T \tau]} u_{\eta,\, t} = \sum_{t=1}^{[T \tau]} \int_{\frac{t}{T}}^{\frac{t+1}{T}} \frac{u_{\eta,\, [T s]}}{\sigma_{\eta} \sqrt{T}} \dd s \liml \int_0^{\tau} W_{r,\, 0}(s) \dd s \equiv C_{r,\, 1}(\tau)
\end{equation}
by application of Theorem \ref{ThmCMT}. Since the leading term of $\whveps_{t}$ is $u_{\eta,\, t}$ as it is explained above and using convergence \eqref{InvP_ResSetaH1}, we get an invariance principle for the partial sum process $(S_{t})$ in \eqref{SumEstRes}, given by
\begin{equation}
\label{InvP_SH1}
\frac{S_{[T \tau]}}{\sigma_{\eta} T^{3/2}} \liml C_{r,\, 1}(\tau).
\end{equation}
We can also reach the same convergence by using Theorem 1 of \cite{MacNeill78} combined with convergence \eqref{InvPRes2}, that is
\begin{equation}
\label{InvP_SumSetaH1}
\frac{1}{\sigma_{\eta} T^{3/2}} \sum_{t=1}^{[T \tau]} S_{t}^{\eta} = \sum_{t=1}^{[T \tau]} \int_{\frac{t}{T}}^{\frac{t+1}{T}} \frac{S_{[T s]}^{\eta}}{\sigma_{\eta} \sqrt{T}} \dd s \liml \int_0^{\tau} W(s) \dd s \equiv W^{(1)}(\tau).
\end{equation}
Of course, \eqref{CvgEstVarH0} cannot hold under $\cHa$ and the asymptotic behavior of $Q_{T}$ will now stem from \eqref{InvP_ResSetaH1}. Indeed,
\begin{equation*}
\frac{1}{\sigma_{\eta}^2\, T^{2}} \sum_{t=1}^{[T \tau]} u_{\eta,\, t}^2 = \sum_{t=1}^{[T \tau]} \int_{\frac{t}{T}}^{\frac{t+1}{T}} \left( \frac{u_{\eta,\, [T s]}}{\sigma_{\eta} \sqrt{T}} \right)^{\! 2} \dd s \liml \int_0^{\tau} W_{r,\, 0}^{\, 2}(s) \dd s
\end{equation*}
implying that 
\begin{equation}
\label{CvgEstVarH1}
\frac{Q_{[T \tau]}}{\sigma_{\eta}^2\, T^{2}} \liml \int_0^{\tau} W_{r,\, 0}^{\, 2}(s) \dd s.
\end{equation}
In addition, from \eqref{InvP_SH1},
\begin{equation*}
\frac{1}{\sigma^2_{\eta}\, T^4} \sum_{t=1}^{[T \tau]} S_{t}^{\, 2} = \frac{1}{T}\, \sum_{t=1}^{[T \tau]} \left( \frac{S_{t}}{\sigma_{\eta} T^{3/2}} \right)^{\! 2} = \sum_{t=1}^{[T \tau]} \int_{\frac{t}{T}}^{\frac{t+1}{T}} \left(  \frac{S_{[T s]}}{\sigma_{\eta} T^{3/2}} \right)^{\! 2} \dd s \liml \int_0^{\tau} C_{r,\, 1}^{\, 2}(s) \dd s.
\end{equation*}
The latter convergence together with \eqref{CvgEstVarH1} and Theorem \ref{ThmCMT} achieve the first part of the proof, by selecting $\tau=1$. On the other hand, for $\rho=-1$, the summation \eqref{InvP_SumSetaH1} is different due to the phenomenon of compensation. As a matter of fact, it is not hard to see that, for any even and odd integer $t \geq 1$, respectively, we have
\begin{equation*}
\sum_{k=1}^{t} S_{k}^{\eta} = \sum_{k=1}^{t/2} \eta_{2k} \hsp \hsp \text{and} \hsp \hsp \sum_{k=1}^{t} S_{k}^{\eta} = \sum_{k=1}^{(t+1)/2} \eta_{2k-1}.
\end{equation*}
Let $(\zeta_{t})$ be the sequence defined, for an even $T$ and all $1 \leq t \leq T/2$, by
\begin{equation*}
\zeta_{t} = \veps_{2 t-1} + \veps_{2 t} + \eta_{2 t}
\end{equation*}
and, for an odd $T$ and all $1 \leq t \leq (T+1)/2$, by
\begin{equation*}
\zeta_{t} = \veps_{2 t-1} + \veps_{2 (t-1)} + \eta_{2 t-1}.
\end{equation*}
Hence, $\dE[\zeta_{t}]=0$, $\dE[\zeta_{t}^2] = 2 \sigeps + \sige$ and all covariances are zero, since $(\veps_{t})$ and $(\eta_{t})$ are not correlated. It follows that $(\zeta_{t})$ is a white noise and that it satisfies, by virtue of Theorem 1 of \cite{DedeckerRio00}, the invariance principle
\begin{equation}
\label{InvPSumSetaH1b}
\frac{1}{\sqrt{T}} \sum_{t=1}^{[T \tau]} \zeta_{t} \liml \sqrt{ 2 \sigeps + \sige } ~ W(\tau).
\end{equation}
Thus, we obtain the invariance principles
\begin{equation*}
\frac{1}{\sqrt{T}} \sum_{t=1}^{[T \tau]} \big( S_{t}^{\eta} + \veps_{t} \big) = \frac{1}{\sqrt{T}} \sum_{t=1}^{[T \tau/2]} \zeta_{t} \liml \sqrt{ 2 \sigeps + \sige } ~ W\!\left( \frac{\tau}{2} \right) \egl \sqrt{ \frac{2 \sigma_{\veps}^2 + \sigma_{\eta}^2}{2} } ~ W(\tau)
\end{equation*}
and, by application of Theorem 1 of \cite{MacNeill78},
\begin{equation}
\label{InvPSumResSetaH1b}
\frac{1}{\sqrt{T}} \sum_{t=1}^{[T \tau]} u_{\eta,\, t} \liml \sqrt{ \frac{2 \sigma_{\veps}^2 + \sigma_{\eta}^2}{2} } ~ B_{r}(\tau).
\end{equation}
Exploiting the latter convergence and the domination of $u_{\eta,\, t}$ in $\whveps_{t}$ (the estimator of $\theta$ remaining consistent), it follows that
\begin{equation}
\label{InvPSumResS2H1b}
\frac{1}{T^2} \sum_{t=1}^{[T \tau]} S_{t}^{\, 2} = \sum_{t=1}^{[T \tau]} \int_{\frac{t}{T}}^{\frac{t+1}{T}} \left( \frac{S_{[T s]}}{\sqrt{T}} \right)^{2} \dd s \liml \frac{2 \sigma_{\veps}^2 + \sigma_{\eta}^2}{2} \int_0^{\tau} B_{r}^{\, 2}(s) \dd s.
\end{equation}
Let us now restart the reasoning developed in Lemma \ref{LemInvPRes}, but for $d=1$ and $\rho=-1$. We recall that, using the notations associated with \eqref{Pk}, for all $1 \leq k \leq r+1$,
\begin{equation*}
P_{k,\, T} = \sum_{t=1}^{T} t_{T}^{k-1} w_{t} = \frac{1}{T^{k-1}} \sum_{t=1}^{T} t^{k-1} \left( S_{t}^{\eta} + \veps_{t} \right).
\end{equation*}
First, it is not hard to see that
\begin{equation*}
M^{k}_{T} = \sum_{t=1}^{T} t^{k-1} \veps_{t}
\end{equation*}
is a martingale adapted to the natural filtration of $(\veps_{t})$, whose increasing process is such that $\langle M^{k} \rangle_{T} = O(T^{2k - 1})$ a.s. with obviously
$$
\limT \langle M^{k} \rangle_{T} = +\infty \cvgps
$$
The law of large numbers for scalar martingales (see \textit{e.g.} \cite{Duflo97}) implies that $M^{k}_{T} = o(T^{k})$ a.s. Hence,
\begin{equation}
\label{OLSH1bT1}
\frac{P_{k,\, T}}{T} = \frac{1}{T^{k}} \sum_{t=1}^{T} t^{k-1} S_{t}^{\eta} ~ + ~ o(1) \cvgps
\end{equation}
In addition, denote by $(\Sigma_{t}^{\eta})$ the partial sum process associated with $(\eta_{t})$ for $\rho = 1$. Let also $(\Lambda_{t}^{\eta})$ and $(\Pi_{t}^{\eta})$ be the partial sum processes associated with $(\eta_{t})$, for the even and odd subscripts, respectively. Explicitly,
\begin{equation*}
\Lambda_{p_{t}}^{\eta} = \eta_{2} + \eta_{4} + \hdots + \eta_{2 p_{t}} = \sum_{\ell=1}^{p_{t}} \eta_{2\ell}
\end{equation*}
and
\begin{equation*}
\Pi_{i_{t}}^{\eta} = \eta_{1} + \eta_{3} + \hdots + \eta_{2 i_{t} - 1} = \sum_{\ell=1}^{i_{t}} \eta_{2\ell-1}
\end{equation*}
with $i_{t} = [(t+1)/2]$ and $p_{t} = t - [(t+1)/2]$. A direct calculation shows that, for $\rho=-1$ and all $1 \leq k \leq r+1$,
\begin{equation}
\label{DecompStH1b}
\sum_{t=1}^{T} t^{k-1} S_{t}^{\eta} = \sum_{t=1}^{T} t^{k-1} \Sigma_{t}^{\eta} - 2 \sum_{t=1}^{p_{T}} (2t+1)^{k-1} \Lambda_{t}^{\eta} - 2 \sum_{t=1}^{i_{T}} (2t)^{k-1} \Pi_{t}^{\eta} + 2\, r_{T}
\end{equation}
where we have $r_{T} = (T+1)^{k-1} \Pi_{(T+1)/2}^{\eta}$ for all odd $T$ and $r_{T}=(T+1)^{k-1} \Lambda_{T/2}$ for all even $T$. It is possible, \textit{via} Theorem \ref{ThmDonsker}, to establish an invariance principle on the processes $(\Lambda_{t}^{\eta})$ and $(\Pi_{t}^{\eta})$. As a matter of fact,
\begin{equation}
\label{DecompStH1bPID}
\frac{\Lambda_{[p_{T} \tau]}^{\eta}}{\sigma_{\eta} \sqrt{p_{T}}} \liml W(\tau) \hsp \text{and} \hsp \frac{\Pi_{[i_{T} \tau]}^{\eta}}{\sigma_{\eta} \sqrt{i_{T}}} \liml W(\tau).
\end{equation}
It follows, from Theorem \ref{ThmCMT}, that
\begin{eqnarray}
\label{DecompStH1bPIDSP}
\frac{1}{\sigma_{\eta}\, p_{T}^{~ k+1/2}} \sum_{t=1}^{[p_{T} \tau]} (2t+1)^{k-1} \Lambda_{t}^{\eta} = \sum_{t=1}^{[p_{T} \tau]} \int_{\frac{t}{p_{T}}}^{\frac{t+1}{p_{T}}} \frac{(2 [p_{T} s]+1 )^{k-1}\, \Lambda_{[p_{T} s]}^{\eta}}{\sigma_{\eta}\, p_{T}^{~ k-1} \sqrt{p_{T}}} \dd s \hspace{2cm} \nonumber \\
\hspace{5cm} \liml \int_0^{\tau} (2 s)^{k-1}\, W(s)\, \dd s
\end{eqnarray}
and that
\begin{eqnarray}
\label{DecompStH1bPIDSI}
\frac{1}{\sigma_{\eta}\, i_{T}^{~ k+1/2}} \sum_{t=1}^{[i_{T} \tau]} (2t)^{k-1} \Pi_{t}^{\eta} = \sum_{t=1}^{[i_{T} \tau]} \int_{\frac{t}{i_{T}}}^{\frac{t+1}{i_{T}}} \frac{(2 [i_{T} s] )^{k-1}\, \Pi_{[i_{T} s]}^{\eta}}{\sigma_{\eta}\, i_{T}^{~ k-1} \sqrt{i_{T}}} \dd s \hspace{2cm} \nonumber \\
\hspace{5cm} \liml \int_0^{\tau} (2 s)^{k-1}\, W(s)\, \dd s
\end{eqnarray}
since it is not hard to see that $p_{T}$ and $i_{T}$ behave like $T/2$. Moreover, the convergences \eqref{DecompStH1bPID} and the definition of $r_{T}$ directly lead to
\begin{equation}
\label{DecompStH1bPIDReste}
\frac{r_{T}}{T^{k+1/2}} \limp 0.
\end{equation}
In addition, the invariance principle \eqref{InvP_SumTkW} for $\rho=1$ and $d=1$, here corresponding to the one associated with $(\Sigma_{t}^{\eta})$, gives, together with \eqref{DecompStH1b}, \eqref{DecompStH1bPIDSP}, \eqref{DecompStH1bPIDSI} and \eqref{DecompStH1bPIDReste},
\begin{equation*}
\frac{1}{T^{k+1/2}} \sum_{t=1}^{T} t^{k-1} S_{t}^{\eta} = O_{\dP}(1)
\end{equation*}
and thus, with the notations above, for all $1 \le k \leq r+1$,
\begin{equation*}
\frac{P_{k,\, T}}{T^{3/2}} = O_{\dP}(1) \hsp \text{and} \hsp \frac{u_{\eta,\, T}}{\sqrt{T}} = \frac{S_{T}^{\eta} + \veps_{T}}{\sqrt{T}} + O_{\dP}(1),
\end{equation*}
successively using \eqref{DiffOLS} and \eqref{DecompResOLS}. By virtue of Theorems \ref{ThmDonsker}--\ref{ThmCMT} and the strong law of large numbers, we deduce, following the same calculations, that the process $(Q_{t})$ grows with rate $T^2$, which achieves the proof for $\rho=-1$ since \eqref{InvPSumResS2H1b} shows that the numerator of $\whk$ also grows with the same rate. Finally, for $\kappa = 0$, the invariance principle \eqref{InvP_ResSetaH1} merely becomes
\begin{equation}
\label{InvP_ResSetaH1NT}
\frac{u_{\eta,\, [T \tau]}}{\sigma_{\eta} \sqrt{T}} \liml W(\tau)
\end{equation}
from Theorem \ref{ThmDonsker}, and the end of the reasoning easily follows as above. \hfill
$\mathbin{\vbox{\hrule\hbox{\vrule height1ex \kern.5em\vrule height1ex}\hrule}}$

\medskip

\noindent \textbf{Proof of Proposition \ref{PropStatH1b}.} This proof will be very succinct since all results have been established in the previous reasonings. Indeed, for $\kappa=0$ and $\rho=-1$, convergence \eqref{InvPSumResS2H1b} becomes
\begin{equation*}
\frac{1}{T^2} \sum_{t=1}^{[T \tau]} S_{t}^2 \liml \sigeps \int_{0}^{\tau} W_{\veps}^{\, 2}(s) \dd s ~ + ~ \frac{\sige}{2} \int_{0}^{\tau} W_{\eta}^{\, 2}(s) \dd s,
\end{equation*}
if we split the limiting distribution into two independent components, so as to easily deal with in the sequel. Without any trend fitted, we also have $u_{\eta,\,t} = S_{t}^{\eta} + \veps_{t}$, for all $1 \leq t \leq T$. It follows that, similarly,
\begin{equation*}
\frac{Q_{[T \tau]}}{T^{2}} \liml \sige \int_0^{\tau} W_{\eta}^{\, 2}(s) \dd s.
\end{equation*}
We achieve the proof by choosing $\tau=1$ and by applying Theorem \ref{ThmCMT}. \hfill
$\mathbin{\vbox{\hrule\hbox{\vrule height1ex \kern.5em\vrule height1ex}\hrule}}$

\bigskip

\noindent \textbf{Acknowledgements.} The author thanks the anonymous Reviewer for his suggestions and constructive comments which helped to improve substantially the paper.

\nocite{*}

\bibliographystyle{apalike}
\bibliography{STATBIB}
\vspace{-0.2cm}

\bigskip

\end{document}